\documentclass[12pt,draftcls,onecolumn]{IEEEtran}%12pt,onecolumn
\IEEEoverridecommandlockouts                              % This command is only
\def\calA{{\mathcal{A}}}

\def\calP{{\mathcal{P}}}

\def\calX{{\mathcal{X}}}

\def\Real{{\mathbb{R}}}

\let\labelindent\relax
\usepackage{epsfig} % for postscript graphics files
\usepackage{epstopdf}

\usepackage{fancyhdr}
\usepackage{indentfirst}
\usepackage{amsmath}
\usepackage{amssymb}
\usepackage{algorithm}
\usepackage{algorithmic}
\usepackage{enumitem}

\usepackage{amsthm}

\usepackage{graphicx}
\usepackage{color}

\definecolor{gray}{RGB}{128,128,128}

\usepackage[short]{optidef}
\usepackage{latexsym}
\usepackage{amsfonts}
\usepackage{amssymb,amsmath,amsfonts}
\usepackage{amsmath,mathrsfs,bm,url,times}
\usepackage[table]{xcolor}
\usepackage{cite}
\usepackage{caption}
\usepackage{subcaption}
\newtheorem{assumption}{Assumption}

\newtheorem{theorem}{Theorem}
\newtheorem{lemma}{Lemma}

\newtheorem{remark}{Remark}
\newtheorem{corollary}{Corollary}

\DeclareMathOperator*{\argmin}{arg\,min}

\DeclareMathOperator{\NetReg}{Net-Reg}
\DeclareMathOperator{\col}{col}
\DeclareMathOperator{\inout}{in}
\DeclareMathOperator{\outin}{out}
\usepackage{array}
\newcolumntype{M}[1]{>{\centering\arraybackslash}m{#1}}
\newcolumntype{N}{@{}m{0pt}@{}}

\usepackage{multirow}

\allowdisplaybreaks
\title{\LARGE \bf Regret and Cumulative Constraint Violation Analysis for Distributed Online Constrained\\ Convex Optimization
}

\author{Xinlei Yi, Xiuxian Li, Tao Yang, Lihua Xie,\\ Tianyou Chai, and Karl H. Johansson% 
	\thanks{This work was supported in part by the Knut and Alice Wallenberg Foundation, in part by the  Swedish Foundation for Strategic Research, in part by the Swedish Research Council, in part by the Ministry of Education of Republic of Singapore under Grant AcRF TIER 1- 2019-T1-001-088 (RG72/19), and in part by the National Natural Science Foundation of China under Grant 62003243, Grant 61991403, Grant 61991404, and Grant 61991400.}
	\thanks{X. Yi and K. H. Johansson are with the Division of Decision and Control Systems, School of Electrical Engineering and Computer Science, KTH Royal Institute of Technology, 10044 Stockholm, Sweden, and also with the Digital Futures, 10044 Stockholm, Sweden.
		{\tt\small \{xinleiy, kallej\}@kth.se}.}%
	\thanks{X. Li is with the Department of Control Science and Engineering, College of Electronics and Information Engineering, and the Shanghai Research Institute for Intelligent Autonomous Systems, Tongji University, Shanghai, P. R. China. {\tt\small xli@tongji.edu.cn}.}
	\thanks{T. Yang and T. Chai are with the State Key Laboratory of Synthetical Automation for Process Industries, Northeastern University, 110819 Shenyang, China. {\tt\small \{yangtao,tychai\}@mail.neu.edu.cn}.}
	\thanks{L. Xie is with School of Electrical and Electronic Engineering,
		Nanyang Technological University, 50 Nanyang Avenue, Singapore 639798. {\tt\small elhxie@ntu.edu.sg}.}
}

\begin{document}

\maketitle
\thispagestyle{plain}
\pagestyle{plain}

\begin{abstract}\label{online_op:Abstract}
This paper considers the distributed online convex optimization problem with time-varying constraints over a network of agents. This is a sequential decision making problem with two sequences of arbitrarily varying convex loss and constraint functions. At each round, each agent selects a decision from  the decision set, and then only a portion of the loss function and a coordinate block of the constraint function at this round are privately revealed to this agent. The goal of the network is to minimize the network-wide loss accumulated over time. Two distributed online algorithms with full-information and bandit feedback are proposed. Both dynamic and static network regret bounds are analyzed for the proposed algorithms, and network cumulative constraint violation is used to measure constraint violation, which excludes the situation that strictly feasible constraints can compensate the effects of violated constraints. In particular, we show that the proposed algorithms achieve $\mathcal{O}(T^{\max\{\kappa,1-\kappa\}})$ static network regret and $\mathcal{O}(T^{1-\kappa/2})$ network cumulative constraint violation, where $T$ is the time horizon and $\kappa\in(0,1)$ is a user-defined trade-off parameter. Moreover, if the loss functions are strongly convex, then the static network regret bound can be reduced to $\mathcal{O}(T^{\kappa})$.  Finally, numerical simulations are provided to illustrate the effectiveness of the theoretical results.

\emph{Index Terms}---Cumulative constraint violation, distributed optimization, online optimization, regret, time-varying constraints.
\end{abstract}
%%%%%%%%%%%%%%%%%%%%%%%%%%%%%%%%%%%%%%%%%%%%%%%%%%%%%%%%%%%%%%%%%%%%%%%%%%%%%%%%
\section{INTRODUCTION}\label{online_opsec:intro}
Online convex optimization is a promising framework for machine learning and has wide applications such as online binary classification \cite{crammer2006online}, dictionary learning \cite{mairal2009online}, and online display advertising \cite{goldfarb2011online}. It can be traced back at least to the 1990's \cite{cesa1996worst,gentile1999linear,gordon1999regret}.
Simply speaking, online convex optimization is a sequential decision making problem with a sequence of arbitrarily varying convex loss functions. At each round, a decision maker selects a decision from the decision/constraint set and then the loss function at this round is revealed. The goal of the decision maker is to minimize the loss accumulated over time. For an online convex optimization algorithm, the standard performance metric is regret, which is the performance gap between the decision sequence induced by the algorithm and a benchmark in hindsight. If the benchmark is the optimal static (dynamic) decision sequence, then regret is called static (dynamic) regret.

Over the past decades, online convex optimization has been extensively studied, e.g.,
\cite{zinkevich2003online,hazan2007logarithmic,agarwal2010optimal,shalev2012online,yi2016tracking,mokhtari2016online,hazan2016introduction,shamir2017optimal,zhang2018adaptive,zhang2018dynamic,shames2019online,
	ijcai2020-731}. %pmlr-v124-zhang20a,
In these studies, the proposed algorithms usually are   projection-based and the results basically ensure that sublinear static regret can be achieved. For example, the projection-based online gradient descent algorithm proposed in \cite{zinkevich2003online} achieved an $\mathcal{O}(\sqrt{T})$ static regret bound, where $T$ is the time horizon. This is a tight bound up to constant factors \cite{hazan2007logarithmic}. The static regret bound can be reduced under more stringent strong convexity conditions on the loss functions \cite{hazan2007logarithmic,shalev2012online,mokhtari2016online,hazan2016introduction}. Note that the projection operator is performed at each round. It could yield heavy computation and/or storage burden when the constraint set is determined by a set of functional constraints. To tackle this challenge, online convex optimization with long term constraints was considered, e.g., \cite{mahdavi2012trading,jenatton2016adaptive,NIPS2018_7852,yu2020lowJMLR,yi2021regret}. In this new problem, the constraints are relaxed to be soft long term constraints. In other words, instead of requiring the constraints to be satisfied at each round, the constraints  should be satisfied in the long term on average. In addition to regret, the other performance metric in this case is constraint violation, which is the violation of the cumulative constraint functions. The problem can be further extended to a more general scenario where the constraint functions are time-varying and revealed to the decision maker after the decision is chosen, e.g.,  \cite{paternain2016online,sun2017safety,yu2017online,neely2017online,chen2017online,
	chen2018bandit,chen2018heterogeneous,cao2019online,pmlr-v97-liakopoulos19a,sadeghi2019online,wei2020online}.

Distributed optimization methods are becoming core aspects of various important applications in view of flexibility and scalability to large-scale datasets and systems, and from the perspective of data privacy and locality \cite{pmlr-v97-koloskova19a}. Motivated by this, a distributed variation of the classic online convex optimization was considered, e.g., \cite{tsianos2012distributed,
	mateos2014distributed,koppel2015saddle,hosseini2016online,
	akbari2017distributed,shahrampour2018distributed,akbari2018individual,zhang2019distributed,
	wan2020projection,carnevale2020distributed}.
%yan2013distributed,hosseini2013online,zhang2017projection,xiong2020privacy,
In this setting, at each round the loss function is decomposed across a network of agents. Each agent selects a decision from  the decision set and then its own portion of the loss function, i.e., the local loss function, at this round is revealed to itself only. The goal of the network is to minimize the network-wide loss accumulated over time, and the performance metric for a distributed algorithm is hence measured by network regret, i.e., the average of all individual regrets. Each agent's individual regret is the difference between the cumulative global losses evaluated at this agent's decision sequence and a benchmark in hindsight. In order to avoid the potential computation and/or storage challenge caused by the projection operator when using projection-based algorithms, distributed online convex optimization with long term constraints was considered in \cite{yuan2017adaptive,yuan2021distributed,yuan2021distributedb}. For this problem, network constraint violation can be similarly defined and also is a performance metric. For example, in \cite{yuan2017adaptive}, an $\mathcal{O}(T^{0.5+\beta})$ static network regret bound and an $\mathcal{O}(T^{1-\beta/2})$ network constraint violation bound were achieved, where $\beta\in(0,0.5)$ is a user-defined parameter which enables the trade-off between these two bounds. 
In \cite{sharma2020distributed}, the above distributed online convex optimization with long term constraints was extended to a more general scenario where the constraint functions are time-varying and at each round only a coordinate block of the constraint function is privately revealed to each agent after its decision is chosen, and an $\mathcal{O}(T^{\max\{a,1-a+b\}}+T^aP_T)$ network regret bound and an $\mathcal{O}(\sqrt{T^{\max\{2-b,2+2b-2a\}}+T^{1+a-b}P_T})$ network constraint violation bound were achieved, where $a,b\in(0,1)$ with $a>b$ are user-defined trade-off parameters and $P_T\ge0$ is the path-length of the benchmark (see Theorem~\ref{online_op:corollaryreg}).
Other forms of distributed variation of the centralized online convex optimization have also been considered, e.g., \cite{raginsky2011decentralized,lee2016coordinate,lee2017stochastic,
	Li2018distributed,yi2020distributed,yi2019distributed,li2020distributed}.

It should be pointed out that the commonly used (network) constraint violation metrics have a potential drawback since they implicitly allow constraint violations at some rounds to be compensated by strictly feasible decisions at other rounds. In order to avoid this drawback, a more strict metric, cumulative constraint violation, i.e., the constraint violation accumulated over time, was considered in \cite{NIPS2018_7852,yi2021regret,yuan2021distributed,yuan2021distributedb}. For example, in \cite{NIPS2018_7852}, the centralized online convex optimization with long term constraints was considered, and an $\mathcal{O}(T^{\max\{\kappa,1-\kappa\}})$ static regret bound and an $\mathcal{O}(T^{1-\kappa/2})$ cumulative constraint violation bound were achieved, where $\kappa\in(0,1)$ is a user-defined trade-off parameter. These results were extended to the distributed setting with quadratic loss functions and linear constraint functions in \cite{yuan2021distributed}, and further to the distributed setting with arbitrary convex loss and constraint functions in \cite{yuan2021distributedb}.

In this paper, same as \cite{sharma2020distributed}, we study the general online constrained convex optimization problem. However, different from \cite{sharma2020distributed}, we adopt a more strict metric, network cumulative constraint violation. Moreover, we consider both full-information and bandit feedback scenarios. The full-information feedback means that the expressions or (sub)gradients of the  loss and constraint functions are revealed, while the bandit feedback means that only the values of the loss and constraint functions are revealed at the sampling instance.
We propose two distributed online algorithms to solve the problem, which have a good property that they do not use the time horizon or any other parameters related to the loss or constraint functions to design the algorithm parameters. We have the following contributions.
\begin{itemize}
	\item To the best of our knowledge, this paper is the first to consider cumulative constraint violation for distributed online convex optimization with time-varying constraints, see Remark~\ref{online_op:remark_problem} for more detailed explanations. 
	\item We show in Theorems~\ref{online_op:corollaryreg} and \ref{online_op:corollaryreg_bandit} that the proposed algorithms achieve an $\mathcal{O}(T^{\max\{\kappa,1-\kappa\}}+T^{\kappa}P_T)$ network regret bound and an $\mathcal{O}(T^{1-\kappa/2})$ network cumulative constraint violation bound. These bounds recover the results achieved by the centralized online algorithms proposed in \cite{zinkevich2003online,yi2016tracking,cao2019online}, and are smaller than the bounds achieved in \cite{sharma2020distributed} although the standard network constraint violation metric rather than the more strict metric was used in \cite{sharma2020distributed}, see Remarks~\ref{online_op:remark_dynamic} and \ref{online_op:remark_dynamic_bandit} for more detailed explanations.
	\item We show in Corollaries~\ref{online_op:theoremstatic} and \ref{online_op:theoremstatic_bandit} that the proposed algorithms achieve an $\mathcal{O}(T^{\max\{\kappa,1-\kappa\}})$ static network regret bound and an $\mathcal{O}(T^{1-\kappa/2})$ network cumulative constraint violation bound. These bounds generalize the results in \cite{agarwal2010optimal,shamir2017optimal,mahdavi2012trading,jenatton2016adaptive,
		NIPS2018_7852,sun2017safety,yuan2017adaptive,yuan2021distributed,yuan2021distributedb} to more general settings,  and also improve the results in \cite{yuan2017adaptive}, see Remarks~\ref{online_op:remark_static} and \ref{online_op:remark_static_bandit} for more detailed explanations.
	\item When the loss functions are strongly convex, we show in Theorems~\ref{online_op:corollaryreg_sc} and \ref{online_op:corollaryreg_sc_bandit} that the static network regret bound can be improved to $\mathcal{O}(T^{\kappa})$. This result generalizes the result in \cite{yuan2017adaptive} to more general settings, see Remark~\ref{online_op:remark_sc} for more detailed explanations. 
\end{itemize}

In conclusion, as explained at the end of Section~\ref{online_opsec:algorithm}, the results in this paper can be viewed as nontrivial extensions of existing results.
The detailed comparison of this paper to some of the related works is summarized in TABLE~\ref{online_op::table}\footnote{In this table, we do not list the dynamic regret since most of the listed works do not consider that.}.

\bgroup
\def\arraystretch{1.4}
\begin{table*}[ht!]
	\caption{Comparison of this paper to related works on online constrained convex optimization.}
	\label{online_op::table}
	%\vskip 0.15in
	\begin{center}
		\begin{small}
\begin{tabular}{M{1.1cm}|M{1.5cm}|M{1.8cm}|M{1.8cm}|M{2.8cm}|M{3.4cm}|M{2.2cm}N}
\hline
Reference&Problem type&Loss functions&Constraint functions&Static regret&Constraint violation&Cumulative constraint violation&\\

\hline
\cite{mahdavi2012trading}&Centralized&Convex&Convex, time-invariant&$\mathcal{O}(\sqrt{T})$& $\mathcal{O}(T^{3/4})$&Not given&\\

\hline
\multirow{2}{*}{\cite{NIPS2018_7852}}&\multirow{2}{*}{Centralized}&Convex
&\multirow{2}{*}{\parbox{1.8cm}{\centering Convex, time-invariant}}&$\mathcal{O}(T^{\max\{\kappa,1-\kappa\}})$& \multicolumn{2}{c}{$\mathcal{O}(T^{1-\kappa/2})$}&\\[7pt]
\cline{3-3}\cline{5-7}
&&Strongly convex&&$\mathcal{O}(\log(T))$&\multicolumn{2}{c}{$\mathcal{O}(\sqrt{\log(T)T})$}&\\

\hline
\cite{sun2017safety}&Centralized&Convex&Convex&$\mathcal{O}(\sqrt{T})$& $\mathcal{O}(T^{3/4})$&Not given&\\

\hline
\multirow{2}{*}{\cite{yuan2017adaptive}}&\multirow{2}{*}{Distributed}
&Convex&\multirow{2}{*}{\parbox{1.8cm}{\centering Convex, time-invariant}}&$\mathcal{O}(T^{\max\{0.5+\beta\}})$& $\mathcal{O}(T^{1-\beta/2})$&Not given&\\[7pt]
\cline{3-3}\cline{5-7}
&&Strongly convex&&$\mathcal{O}(T^{\kappa})$& $\mathcal{O}(T^{1-\kappa/2})$&Not given&\\

\hline
\cite{yuan2021distributed}&Distributed&Quadratic
&Linear, time-invariant&$\mathcal{O}(T^{\max\{\kappa,1-\kappa\}})$& \multicolumn{2}{c}{$\mathcal{O}(T^{1-\kappa/2})$}&\\

\hline
\multirow{2}{*}{\cite{yuan2021distributedb}}&\multirow{2}{*}{Distributed}&Convex
&\multirow{2}{*}{\parbox{1.8cm}{\centering Convex, time-invariant}}&$\mathcal{O}(T^{\max\{\kappa,1-\kappa\}})$& \multicolumn{2}{c}{$\mathcal{O}(T^{1-\kappa/2})$}&\\
\cline{3-3}\cline{5-7}
&&Strongly convex&&$\mathcal{O}(\log(T))$&\multicolumn{2}{c}{$\mathcal{O}(\sqrt{\log(T)T})$}&\\

\hline
\cite{sharma2020distributed}&Distributed&Convex&Convex&$\mathcal{O}(T^{\max\{a,1-a+b\}})$& $\mathcal{O}(T^{\max\{1-b/2,1+b-a\}})$&Not given&\\

\hline
\multirow{2}{*}{\parbox{1.1cm}{\centering This paper}}&\multirow{2}{*}{Distributed}&Convex
&\multirow{2}{*}{Convex}&$\mathcal{O}(T^{\max\{\kappa,1-\kappa\}})$& \multicolumn{2}{c}{\multirow{2}{*}{$\mathcal{O}(T^{1-\kappa/2})$}}&\\
\cline{3-3}\cline{5-5}
&&Strongly convex&&$\mathcal{O}(T^\kappa)$&\multicolumn{2}{c}{}&\\
\hline
\end{tabular}

\end{small}
\end{center}
\vskip -0.1in
\end{table*}
\egroup

The rest of this paper is organized as follows. Section~\ref{online_opsec:problem} formulates the considered problem. Sections~\ref{online_opsec:algorithm} and \ref{online_opsec:algorithm_bandit} provide distributed online algorithms with full-information and bandit feedback, respectively, and analyze their regret and cumulative constraint violation bounds. Section~\ref{online_opsec:simulation} gives simulation examples. Finally, Section~\ref{online_opsec:conclusion} concludes this paper.

\noindent {\bf Notations}: All inequalities and equalities throughout this paper are understood componentwise. $\mathbb{R}^p$ and $\mathbb{R}^p_+$ stand for the set of $p$-dimensional vectors and nonnegative vectors, respectively. $\mathbb{N}_+$ denotes the set of all positive integers. $[n]$ represents the set $\{1,\dots,n\}$ for any positive integer $n$. $\|\cdot\|$ ($\|\cdot\|_1$) stands for the Euclidean norm (1-norm) for vectors and the induced 2-norm (1-norm) for matrices. $\mathbb{B}^p$ and $\mathbb{S}^p$ are the unit ball and sphere centered around the origin in $\mathbb{R}^p$ under Euclidean norm, respectively. 
$x^\top$ denotes the transpose of a vector or a matrix. $\langle x,y\rangle$ represents the standard inner product of two vectors $x$ and $y$. ${\bf 0}_m$ (${\bf 1}_m$) denotes the column zero (one) vector with dimension $m$. $\col(z_1,\dots,z_n)$ is the concatenated column vector of $z_i\in\mathbb{R}^{m_i},~i\in[n]$. The notation $A\otimes B$ denotes the Kronecker product
of matrices $A$ and $B$. For a set $\mathbb{K}\subseteq\mathbb{R}^p$, $\calP_{\mathbb{K}}(\cdot)$ denotes the projection operator, i.e.,  $\calP_{\mathbb{K}}(x)=\argmin_{y\in\mathbb{K}}\|x-y\|^2,~\forall x\in\Real^{p}$. For simplicity, $[\cdot]_+$ is used to denote $\calP_{\mathbb{R}^p_+}(\cdot)$. For a scalar function $f:\mathbb{R}^p\rightarrow\mathbb{R}$, let $\partial f(x)\in\mathbb{R}^p$ denote the (sub)gradient of $f$ at $x$, and let $\partial [f(x)]_+$ denote the (sub)gradient of $[f]_+$ at $x$, i.e.,
\begin{align*}
	\partial [f(x)]_+=
	\begin{cases}
		\bm{0}_p, & \mbox{if } f(x)<0 \\
		\partial f(x), & \mbox{otherwise}.
	\end{cases}
\end{align*}
For a vector function $f=[f_1,\dots,f_d]^\top:\mathbb{R}^{p}\rightarrow\mathbb{R}^d$, its (sub)gradient at $x$ is written as $\partial f(x)=[\partial f_1(x),\dots,\partial f_d(x)]\in\mathbb{R}^{p\times d}$. Similarly, the (sub)gradient of $[f]_+$ at $x$ is written as $\partial [f(x)]_+=[\partial [f_1(x)]_+,\dots,\partial [f_d(x)]_+]\in\mathbb{R}^{p\times d}$.

\section{Problem Formulation}\label{online_opsec:problem}
In this section, we formulate the considered problem and provide the motivating examples at the end of this section.

We consider distributed online convex optimization with time-varying constraints. Specifically, consider a network of $n$ agents indexed by $i\in[n]$, which can communicate with each other according to a time-varying directed graph which will be described shortly. Let $\{f_{i,t}:\mathbb{X}\rightarrow \mathbb{R}\}$ and $\{g_{i,t}:\mathbb{X}\rightarrow \mathbb{R}^{m_i}\}$ be sequences of local convex loss and constraint functions over time $t=1,2,\dots$, respectively, where $\mathbb{X}\subseteq\mathbb{R}^p$ is a known convex set, $p$ and $m_i$ are positive integers, and $g_{i,t}\le{\bm 0}_{m_i}$ is the local constraint. At time $t$, each agent $i$ selects a decision $x_{i,t}\in \mathbb{X}$. After the selection, the agent receives full-information or bandit feedback about the loss function $f_{i,t}$ and constraint function $g_{i,t}$, which is held privately by this agent. 
The objective is to design distributed sequential decision selection algorithms such that the network-wide loss accumulated over time is minimized. Similar to \cite{yuan2017adaptive,yuan2021distributed,yuan2021distributedb,sharma2020distributed}, we use network regret and cumulative constraint violation to measure performance of such an algorithm. Specifically, network regret and cumulative constraint violation are defined as
\begin{align}\label{online_op:reg}
	\NetReg(\{x_{i,t}\},y_{[T]})
	:=&\frac{1}{n}\sum_{i=1}^{n}\sum_{t=1}^{T}f_t(x_{i,t})-\sum_{t=1}^{T}f_t(y_{t})
\end{align}
and
\begin{align}\label{online_op:regc}
	\frac{1}{n}\sum_{i=1}^n\sum_{t=1}^T\|[g_{t}(x_{i,t})]_+\|
\end{align}
respectively, where $T$ is the time horizon, $y_{[T]}=(y_{1},\dots,y_{T})$ is a  benchmark, and
\begin{align}
	f_t(x)&=\frac{1}{n}\sum_{j=1}^nf_{j,t}(x)\label{online_op:ft}\\
	g_t(x)&=\col(g_{1,t}(x),\dots,g_{n,t}(x))\label{online_op:gt}
\end{align}
are the global loss and constraint functions, respectively.

In the literature, there are two commonly used benchmarks. One is the optimal dynamic decision sequence $$x_{[T]}^*=(x^*_{1},\dots,x^*_{T}),$$ where $x_t^*\in\mathbb{X}$ denotes the minimizer of $f_t(x)$ constrained by $g_t(x)\le{\bm 0}_{m}$ with $m=\sum_{i=1}^{n}m_i$. In other words, $x_t^*$ is the best choice by knowing the functions $f_{t}$ and $g_{t}$ in advance.
In order to guarantee that the optimal dynamic decision sequence always exists, we assume that for any $T\in\mathbb{N}_+$, the set of all the feasible decision sequences
$$
\mathcal{X}_{T}=\{(x_1,\dots,x_T):~x_t\in \mathbb{X},~g_{t}(x_t)\le{\bf0}_{m},~
\forall t\in[T]\}
$$ is non-empty. In this case $\NetReg(\{x_{i,t}\},x_{[T]}^*)$ is called the dynamic network regret. Another benchmark is the optimal static decision sequence $$\check{x}^*_{[T]}=(\check{x}^*_T,\dots,\check{x}^*_T),$$ where $\check{x}^*_T\in\mathbb{X}$ denotes the minimizer of $\sum_{t=1}^{T}f_t(x)$ constrained by $g_t(x)\le{\bm 0}_{m},~\forall t\in[T]$.
Similar to above, in order to guarantee that the optimal static decision sequence always exists, we assume that for any $T\in\mathbb{N}_+$, the set of all the feasible static decision sequences
$$\check{\calX}_{T}=\{(x,\dots,x):~x\in \mathbb{X},~g_{t}(x)\le{\bf0}_{m},~\forall t\in[T]\}\subseteq\calX_{T}$$ is non-empty.
In this case $\NetReg(\{x_{i,t}\},\check{x}^*_{[T]})$ is called the static network regret. 
The network cumulative constraint violation $\frac{1}{n}\sum_{i=1}^n\sum_{t=1}^T\|[g_{t}(x_{i,t})]_+\|$ is more strict than the network constraint violation $\frac{1}{n}\sum_{i=1}^n\|[\sum_{t=1}^Tg_{t}(x_{i,t})]_+\|$ since
\begin{align*}
	\frac{1}{n}\sum_{i=1}^n\Big\|\Big[\sum_{t=1}^Tg_{t}(x_{i,t})\Big]_+\Big\|
	\le\frac{1}{n}\sum_{i=1}^n\Big\|\sum_{t=1}^T[g_{t}(x_{i,t})]_+\Big\|
	\le\frac{1}{n}\sum_{i=1}^n\sum_{t=1}^T\|[g_{t}(x_{i,t})]_+\|.
\end{align*}
For simplicity purposes, we use standard constraint violation metrics to refer the metrics that take the summation over rounds before the projection, such as the network constraint violation $\frac{1}{n}\sum_{i=1}^n\|[\sum_{t=1}^Tg_{t}(x_{i,t})]_+\|$.
The standard constraint violation metrics are commonly used in the literature, e.g., \cite{mahdavi2012trading,jenatton2016adaptive,sun2017safety,chen2017online,
	neely2017online,yu2017online,chen2018heterogeneous,chen2018bandit,cao2019online,
	yuan2017adaptive,sharma2020distributed}, but have the drawback that they implicitly allow constraint violations at some rounds to be compensated by strictly feasible decisions at other rounds. It is straightforward to see that the network cumulative constraint violation $\frac{1}{n}\sum_{i=1}^n\sum_{t=1}^T\|[g_{t}(x_{i,t})]_+\|$ does not have such a drawback.

Note that each agent alone cannot compute network regret and cumulative constraint violation since it does not know other agents' local loss and constraint functions. Agents can use a consensus protocol to collaborate. Therefore, agents need to communicate with each other. We assume that agents are allowed to share their decisions through a communication network modeled by a time-varying directed graph. Specifically, let $\mathcal{G}_t=(\mathcal{V},\mathcal{E}_t)$ denote the directed graph at the $t$-th round, where $\mathcal{V}=[n]$ is the agent set and $\mathcal{E}_t\subseteq\mathcal{V}\times\mathcal{V}$ the edge set. A directed edge $(j,i)\in\mathcal{E}_t$ means that agent $i$ can receive data from agent $j$ at the $t$-th round. Let $\mathcal{N}^{\inout}_i(\mathcal{G}_t)=\{j\in [n]\mid (j,i)\in\mathcal{E}_t\}$ and $\mathcal{N}^{\outin}_i(\mathcal{G}_t)=\{j\in [n]\mid (i,j)\in\mathcal{E}_t\}$ be the sets of in- and out-neighbors, respectively, of agent $i$ at the $t$-th round. A directed path is a sequence of consecutive directed edges. A  directed graph is said to be strongly connected if there is at least one directed path
from any agent to any other agent in the graph. The adjacency (mixing) matrix $W_t\in\mathbb{R}^{n\times n}$  fulfills $[W_t]_{ij}>0$ if $(j,i)\in\mathcal{E}_t$ or $i=j$, and $[W_t]_{ij}=0$ otherwise.

We make the following standing assumption on the loss and constraint functions.
\begin{assumption}\label{online_op:assfunction}
	\begin{enumerate}
		\item The set $\mathbb{X}$ is convex and closed. Moreover, it contains the ball of radius $r(\mathbb{X})$ centered at the origin and is contained in the ball of radius $R(\mathbb{X})$, i.e.,
		\begin{align}
			r(\mathbb{X})\mathbb{B}^p\subseteq\mathbb{X}\subseteq R(\mathbb{X})\mathbb{B}^p.\label{online_op:domainupper}
		\end{align}
		\item For any $i\in[n],~t\in\mathbb{N}_+$, the functions $f_{i,t}$ and $g_{i,t}$ are convex.  Moreover, there exists a constant $F_1$ such that
		\begin{subequations}\label{online_op:ftgtupper}
			\begin{align}
				&|f_{i,t}(x)-f_{i,t}(y)|\le F_1,\label{online_op:ftgtupper-a}\\
				&\|g_{i,t}(x)\|\le F_1,~\forall i\in[n],~t\in\mathbb{N}_+,~ x,y\in\mathbb{X}.\label{online_op:ftgtupper-b}
			\end{align}
		\end{subequations}
		\item For any $i\in[n],~t\in\mathbb{N}_+,~x\in\mathbb{X}$, the subgradients $\partial f_{i,t}(x)$ and $\partial g_{i,t}(x)$ exist. Moreover, there exists a constant $F_2$ such that
		\begin{subequations}\label{online_op:subgupper}
			\begin{align}
				&\|\partial f_{i,t}(x)\|\le F_2,\label{online_op:subgupper-a}\\
				& \|\partial g_{i,t}(x)\|\le F_2,~\forall i\in[n],~t\in\mathbb{N}_+,~ x\in\mathbb{X}.\label{online_op:subgupper-b}
			\end{align}
		\end{subequations}
	\end{enumerate}
\end{assumption}

The following assumption is made on the graph.
\begin{assumption}\label{online_op:assgraph}
	For any $t\in\mathbb{N}_+$, the directed graph $\mathcal{G}_t$ satisfies the following conditions:
	\begin{enumerate}[label=(\alph*)]
		\item There exists a constant $w\in(0,1)$, such that $[W_t]_{ij}\ge w$ if $[W_t]_{ij}>0$.
		\item The mixing matrix $W_t$ is doubly stochastic, i.e., $\sum_{i=1}^n[W_t]_{ij}=\sum_{j=1}^n[W_t]_{ij}=1,~\forall i,j\in[n]$.
		\item There exists an integer $B>0$ such that the directed graph $(\mathcal{V},\cup_{l=0}^{ B -1}\mathcal{E}_{t+l})$ is strongly connected.
	\end{enumerate}
\end{assumption}

\begin{remark}\label{online_op:remark_problem}
	To the best of our knowledge, this paper is the first to consider cumulative constraint violation for distributed online convex optimization with time-varying constraints.
	The problem considered in this paper is a distributed variation of the centralized online convex optimization with time-varying constraints considered in \cite{sun2017safety,chen2017online,yu2017online,neely2017online,chen2018heterogeneous,chen2018bandit,cao2019online}. The same distributed online constrained convex optimization problem had also been considered in \cite{sharma2020distributed}, but in \cite{sharma2020distributed} the standard  network constraint violation metric was used.
	It should be pointed out that the considered problem is more general than the distributed problems considered in \cite{yuan2017adaptive,yuan2021distributed,yuan2021distributedb}. Specifically, in \cite{yuan2017adaptive,yuan2021distributedb} the global constraint function is time-invariant and known by each agent in advance, and in \cite{yuan2017adaptive} the standard network constraint violation metric was used. In \cite{yuan2021distributed}, the local loss functions are quadratic and the global constraint function is time-invariant, linear, and known by each agent in advance.
	It should also be highlighted that the considered problem in this paper and the distributed online optimization with time-varying coupled inequality constraints considered in \cite{yi2020distributed,yi2019distributed} are different kinds of distributed problems. Specifically, in \cite{yi2020distributed,yi2019distributed} at the $t$-th round the global loss function is $\sum_{i=1}^{n}f_{i,t}(x_i)$ and the constraints are $\sum_{i=1}^{n}g_{i,t}(x_i)\le{\bm 0}_{m}$, where $x_i\in\mathbb{R}^{p_i}$ with $p_i$ being a positive integer. Therefore, the algorithms proposed in \cite{yi2020distributed,yi2019distributed} cannot be used to solve the problem considered in this paper. Moreover, the standard constraint violation metric was used in \cite{yi2020distributed,yi2019distributed} and it is unclear how to extend the results to the more strict constraint violation metric.
\end{remark}

Noting that the problem considered in this paper incorporates the problems considered in \cite{sun2017safety,chen2017online,yu2017online,neely2017online,chen2018heterogeneous,
	chen2018bandit,cao2019online,yuan2017adaptive,yuan2021distributed,yuan2021distributedb}, the examples studied in these papers, such as online job scheduling \cite{yu2017online}, online network resource allocation \cite{chen2017online}, mobile fog computing in IoT \cite{chen2018bandit}, online linear regressions \cite{yuan2021distributed}, and online spam filtering task \cite{yuan2021distributedb}, motivate this paper. We omit the details of these motivating examples due to space limitations. In the simulations we use the distributed online linear regression problem with time-varying linear inequality constraints as an example to compare the performance of different algorithms.

\section{Distributed Online Algorithm with Full-Information Feedback}\label{online_opsec:algorithm}
In this section, we  consider  the distributed online constrained convex optimization problem formulated in Section \ref{online_opsec:problem} with full-information feedback. We first propose a distributed online algorithm, and then derive network regret and cumulative constraint violation bounds for this algorithm.

\subsection{Algorithm Description}

Recall that at the $t$-th round, the global loss and constraint functions are given in \eqref{online_op:ft} and \eqref{online_op:gt}, respectively. The associated regularized convex-concave function is
\begin{align*}
	\calA_{t}(x_t,q_t):=f_{t}(x_t)
	+q_t^\top [g_t(x_t)]_+-\frac{\beta_{t+1}}{2}\|q_t\|^2,
\end{align*}
where $x_t\in\mathbb{X}$ and $q_t\in\mathbb{R}^m_+$ represent the primal and dual variables, respectively, and $\beta_{t+1}$ is the regularization parameter. Here, the clipped constraint function $[g_t(x_t)]_+$ is used, which is essential for analyzing cumulative constraint violation. The primal and dual variables can be updated by the standard primal--dual algorithm
\begin{align}
	x_{t+1}&=\calP_\mathbb{X}\Big(x_t-\alpha_{t+1}\frac{\partial \calA_{t}(x_t,q_t)}{\partial x}\Big)
	=\calP_\mathbb{X}(x_t-\alpha_{t+1}\omega_{t+1}),\label{online_op:al_c_x}\\
	q_{t+1}&=\Big[q_t+\gamma_{t+1}\frac{\partial \calA_{t}(x_t,q_t)}{\partial q}\Big]_+
	=[(1-\beta_{t+1}\gamma_{t+1})q_t+\gamma_{t+1}[g_t(x_t)]_+]_+,\label{online_op:al_c_q}
\end{align}
where $\alpha_{t+1}>0$ and $\gamma_{t+1}>0$ are the stepsizes used in the primal and dual updates, respectively, and
\begin{align*}
	\omega_{t+1}=\frac{1}{n}\sum_{i=1}^{n}\partial f_{i,t}(x_{t})+\partial [g_{t}(x_{t})]_+ q_{t}.
\end{align*}

We then modify the centralized algorithm \eqref{online_op:al_c_x}--\eqref{online_op:al_c_q} to a distributed manner. We use $x_{i,t}$ to denote the local copy of the primal variable $x_t$ and rewrite the dual variable in an agent-wise manner, i.e.,
$q_t=\col(q_{1,t},\dots,q_{n,t})$
with each $q_{i,t}\in\mathbb{R}^{m_i}$. Then, for each agent $i$, $z_{i,t+1}$ computed by the consensus protocol \eqref{online_op:al_z} is used to track the average estimation $\frac{1}{n}\sum_{i=1}^nx_{i,t}$ and thus can be used to estimate $x_{t}$. Moreover, $\omega_{i,t+1}$ defined in \eqref{online_op:al_bigomega} can be understood as a part of $\omega_{t+1}$ that is available to agent $i$. In this case,  each $x_{i,t+1}$ is updated by \eqref{online_op:al_x} which is similar to the updating rule \eqref{online_op:al_c_x}, and
the updating rule \eqref{online_op:al_c_q} can be executed in an agent-wise manner as
\begin{align}
	q_{i,t+1}=[(1-\beta_{t+1}\gamma_{t+1})q_{i,t}+\gamma_{t+1}[g_{i,t}(x_{i,t})]_+]_+.
	\label{online_op:al_c_q2}
\end{align}
In order to avoid using the upper bounds of the loss and constraint functions and their subgradients to design the algorithm parameters $\alpha_t$, $\beta_t$, and $\gamma_t$, inspired by the algorithms proposed in \cite{yu2017online,yi2020distributed,neely2017online}, we slightly modify the dual updating rule \eqref{online_op:al_c_q2} as \eqref{online_op:al_q}.
As a result, the updating rule \eqref{online_op:al_c_x}--\eqref{online_op:al_c_q}  can be executed in a distributed manner, which is given in pseudo-code as Algorithm~\ref{online_op:algorithm}.

\begin{algorithm}
	\caption{Distributed Online Algorithm with Full-Information Feedback}
	\begin{algorithmic}\label{online_op:algorithm}
		\STATE \textbf{Input}:   non-increasing and positive sequences $\{\alpha_t\}$, $\{\beta_t\}$ and $\{\gamma_t\}$.
		\STATE \textbf{Initialize}:  $x_{i,1}\in\mathbb{X}$ and $q_{i,1}={\bm 0}_{m_i}$ for all $i\in[n]$.
		\FOR{$t=1,\dots$}
		\FOR{$i=1,\dots,n$  in parallel}
		\STATE  Broadcast $x_{i,t}$ to $\mathcal{N}^{\outin}_i(\mathcal{G}_{t})$ and receive $x_{j,t}$ from $j\in\mathcal{N}^{\inout}_i(\mathcal{G}_{t})$;
		\STATE  Observe $\partial f_{i,t}(x_{i,t})$, $\partial g_{i,t}(x_{i,t})$, and $g_{i,t}(x_{i,t})$;
		\STATE  Update \begin{align}
			z_{i,t+1}&=\sum_{j=1}^n[W_{t}]_{ij}x_{j,t},\label{online_op:al_z}\\
			\omega_{i,t+1}&=\partial f_{i,t}(x_{i,t})+\partial [g_{i,t}(x_{i,t})]_+ q_{i,t},\label{online_op:al_bigomega}\\
			x_{i,t+1}&=\calP_\mathbb{X}(z_{i,t+1}-\alpha_{t+1}\omega_{i,t+1}),\label{online_op:al_x}\\
			q_{i,t+1}&=[(1-\beta_{t+1}\gamma_{t+1})q_{i,t}+\gamma_{t+1}([g_{i,t}(x_{i,t})]_+
			+(\partial [g_{i,t}(x_{i,t})]_+)^\top(x_{i,t+1}-x_{i,t}))]_{+}.\label{online_op:al_q}
		\end{align}
		\ENDFOR
		\ENDFOR
		\STATE  \textbf{Output}: $\{x_{i,t}\}$.
	\end{algorithmic}
\end{algorithm}

\begin{remark}
	Algorithm~\ref{online_op:algorithm} can be recognized as the distributed variant of the centralized algorithm proposed in \cite{yu2017online,neely2017online}. It is different from the distributed algorithm proposed in \cite{yi2020distributed} since they are designed for solving different problems as explained in Remark~\ref{online_op:remark_problem}. One obvious difference is that in our Algorithm~\ref{online_op:algorithm} the local primal variables are communicated between agents, while in the algorithm proposed in \cite{yi2020distributed} the local dual variables are communicated between agents. Therefore, the proofs are also different when analyzing their performance.
\end{remark}

\subsection{Performance Analysis}
This section analyzes network regret and cumulative constraint violation bounds for Algorithm~\ref{online_op:algorithm}.

We first characterize dynamic network regret and cumulative constraint violation bounds based on some natural vanishing stepsize sequences in the following theorem.

\begin{theorem}\label{online_op:corollaryreg}
	Suppose Assumptions~\ref{online_op:assfunction}--\ref{online_op:assgraph} hold. For all $i\in[n]$, let $\{x_{i,t}\}$ be the sequences generated by Algorithm~\ref{online_op:algorithm} with
	\begin{align}\label{online_op:stepsize1}
		\alpha_t=\frac{\alpha_0}{t^{\kappa}},~\beta_t=\frac{1}{t^\kappa},
		~\gamma_t=\frac{1}{t^{1-k}},~\forall t\in\mathbb{N}_+,
	\end{align} where $\alpha_0>0$ and $\kappa\in(0,1)$. Then, for any $T\in\mathbb{N}_+$ and any benchmark $y_{[T]}\in\calX_{T}$,
	\begin{align}
		&\NetReg(\{x_{i,t}\},y_{[T]})
		=\mathcal{O}\Big(\alpha_0 T^{1-\kappa}+\frac{T^\kappa(1+P_T)}{\alpha_0}\Big),\label{online_op:corollaryregequ1}\\
		&\frac{1}{n}\sum_{i=1}^n\sum_{t=1}^T\|[g_{t}(x_{i,t})]_+\|=
		\mathcal{O}(\sqrt{(\alpha_0+1)T^{2-\kappa}}),\label{online_op:corollaryconsequ}
	\end{align}
	where $P_T=\sum_{t=1}^{T-1}\|y_{t+1}-y_{t}\|$ is the path-length of the benchmark $y_{[T]}$.
\end{theorem}
\begin{proof}
	The explicit expressions of the right-hand sides of \eqref{online_op:corollaryregequ1}--\eqref{online_op:corollaryconsequ}, and the proof are given in  Appendix~\ref{online_op:corollaryregproof}.
\end{proof}

\begin{remark}\label{online_op:remark_dynamic}
	It should be pointed out that the sequences in \eqref{online_op:stepsize1} do not use the time horizon $T$ or any other parameters related to the loss or constraint functions. The intuition of designing the sequences in \eqref{online_op:stepsize1} is to make the network regret and cumulative constraint violation bounds provided in Lemma~\ref{online_op:theoremreg} in  Appendix~\ref{online_op:corollaryregproof} as small as possible. The idea is original to \cite{yi2020distributed} and has also been used in \cite{sharma2020distributed,yi2019distributed}.
	The omitted constants in the right-hand sides of \eqref{online_op:corollaryregequ1}--\eqref{online_op:corollaryconsequ} depend on the user-defined trade-off parameter $\kappa$, the number of agents $n$, the constants related to the loss and constraint functions as assumed in Assumption~\ref{online_op:assfunction}, and the constants related to the communication network connectivity as assumed in Assumption~\ref{online_op:assgraph}. Theorem~\ref{online_op:corollaryreg} shows that Algorithm~\ref{online_op:algorithm} achieves improved performance compared with the dynamic network regret bound $\mathcal{O}(T^{\max\{a,1-a+b\}}+T^aP_T)$ and the standard network constraint violation bound $\mathcal{O}(\sqrt{T^{\max\{2-b,2+2b-2a\}}+T^{1+a-b}P_T})$ achieved by the distributed online algorithm proposed in \cite{sharma2020distributed}, where $a,b\in(0,1)$ and $a>b$.
	If setting $\alpha_0=1$ and  $\kappa=0.5$ in Theorem~\ref{online_op:corollaryreg}, the dynamic regret bound $\mathcal{O}(\sqrt{T}(1+P_T))$ for centralized online convex optimization achieved in \cite{zinkevich2003online} is recovered. If the path-length $P_T$ is known in advance, then setting $\alpha_0=\sqrt{1+P_T}$ and  $\kappa=0.5$ in Theorem~\ref{online_op:corollaryreg} recovers the dynamic regret bound $\mathcal{O}(\sqrt{T(1+P_T)})$. This is the optimal dynamic regret bound for centralized online convex optimization as shown in \cite{zhang2018adaptive,yi2021regret}.
\end{remark}

We then provide static network regret and cumulative constraint violation bounds for Algorithm~\ref{online_op:algorithm} by
replacing $y_{[T]}$ with the static sequence $\check{x}^*_{[T]}$ in Theorem~\ref{online_op:corollaryreg}.
\begin{corollary}\label{online_op:theoremstatic}
	Under the same conditions as stated in Theorem~\ref{online_op:corollaryreg} with $\alpha_0=1$, for any $T\in\mathbb{N}_+$, it holds that
	\begin{align}
		&\NetReg(\{x_{i,t}\},\check{x}^*_{[T]})
		=\mathcal{O}(T^{\max\{\kappa,1-\kappa\}}),\label{online_op:staticregequ1}\\
		&\frac{1}{n}\sum_{i=1}^n\sum_{t=1}^T\|[g_{t}(x_{i,t})]_+\|=
		\mathcal{O}(T^{1-\kappa/2}).
		\label{online_op:staticconsequ}
	\end{align}
\end{corollary}

\begin{remark}\label{online_op:remark_static}
	The results presented in Corollary~\ref{online_op:theoremstatic} generalize the results in \cite{jenatton2016adaptive,NIPS2018_7852,sun2017safety,yuan2021distributed}.
	Specifically,  by setting $\kappa=0.5$ in Corollary~\ref{online_op:theoremstatic}, the result in \cite{sun2017safety} is recovered, although the algorithm proposed in \cite{sun2017safety} is centralized and the standard constraint violation metric rather than the more strict metric is used. The bounds presented in \eqref{online_op:staticregequ1}--\eqref{online_op:staticconsequ} are consistent with the result in \cite{jenatton2016adaptive,NIPS2018_7852,yuan2021distributedb}, although the proposed algorithm in \cite{jenatton2016adaptive,NIPS2018_7852} is centralized, the constraint functions in \cite{jenatton2016adaptive,NIPS2018_7852,yuan2021distributedb} are time-invariant and known in advance, and the standard constraint violation metric is used in \cite{jenatton2016adaptive}.
	The same performance was also achieved in \cite{yuan2021distributed} when the loss functions are quadratic and the constraint functions are time-invariant, linear, and known in advance.
	Corollary~\ref{online_op:theoremstatic} also shows that Algorithm~\ref{online_op:algorithm} achieves improved performance compared with the static network regret bound $\mathcal{O}(T^{0.5+\beta})$ and the standard network constraint violation bound $\mathcal{O}(T^{1-\beta/2})$ achieved by the distributed online algorithm proposed in \cite{yuan2017adaptive}, where $\beta\in(0,0.5)$, although the global constraint functions in \cite{yuan2017adaptive} are time-invariant and known in advance by each agent. 
\end{remark}

The static network regret bound in Corollary~\ref{online_op:theoremstatic} at least is $\mathcal{O}(\sqrt{T})$ and it can be reduced to strictly less than $\mathcal{O}(\sqrt{T})$ if the local loss functions $f_{i,t}(x)$ are strongly convex.
\begin{assumption}\label{online_op:assstrongconvex}
	For any $i\in[n]$ and $t\in\mathbb{N}_+$, $\{f_{i,t}(x)\}$ are strongly convex with convex parameter $\mu>0$ over $\mathbb{X}$ , i.e., for all $x,y\in\mathbb{X}$,
	\begin{align}\label{online_op:assstrongconvexequ}
		f_{i,t}(x)\ge f_{i,t}(y)+\langle x-y,\partial f_{i,t}(y)\rangle+\frac{\mu}{2}\|x-y\|.
	\end{align}
\end{assumption}

\begin{theorem}\label{online_op:corollaryreg_sc}
	Suppose Assumptions~\ref{online_op:assfunction}--\ref{online_op:assstrongconvex} hold. For all $i\in[n]$, let $\{x_{i,t}\}$ be the sequences generated by Algorithm~\ref{online_op:algorithm} with
	\begin{align}\label{online_op:stepsize1_sc}
		\alpha_t=\frac{1}{t^{c}},~\beta_t=\frac{1}{t^\kappa},
		~\gamma_t=\frac{1}{t^{1-k}},~\forall t\in\mathbb{N}_+,
	\end{align} where $c\in[\max\{\kappa,1-\kappa\},1)$ and $\kappa\in(0,1)$. Then, for any $T\in\mathbb{N}_+$, it holds that
	\begin{align}
		&\NetReg(\{x_{i,t}\},\check{x}^*_{[T]})
		=\mathcal{O}(T^{\kappa}),\label{online_op:corollaryregequ1_sc}\\
		&\frac{1}{n}\sum_{i=1}^n\sum_{t=1}^T\|[g_{t}(x_{i,t})]_+\|=
		\mathcal{O}(T^{1-\kappa/2}).
		\label{online_op:corollaryconsequ_sc}
	\end{align}
\end{theorem}
\begin{proof}
	The explicit expressions of the right-hand sides of \eqref{online_op:corollaryregequ1_sc}--\eqref{online_op:corollaryconsequ_sc}, and the proof are given in  Appendix~\ref{online_op:corollaryregproof_sc}.
\end{proof}
\begin{remark}\label{online_op:remark_sc}
	Theorem~\ref{online_op:corollaryreg_sc} shows that under the strongly convex assumption Algorithm~\ref{online_op:algorithm} achieves the same static network regret and constraint violation bounds as the distributed algorithm proposed in \cite{yuan2017adaptive}. However, in \cite{yuan2017adaptive} the standard constraint violation metric rather than the more strict metric is used, and the global constraint functions are time-invariant and known in advance by each agent. Moreover, in \cite{yuan2017adaptive} the convex parameter and the upper bound of the subgradients of the loss and constraint functions need to be known in advance to design the algorithm parameters. However, the bounds presented in \eqref{online_op:corollaryregequ1_sc}--\eqref{online_op:corollaryconsequ_sc} are larger than the bounds achieved by the centralized and distributed algorithms respectively proposed in \cite{NIPS2018_7852} and \cite{yuan2021distributedb} for strongly convex loss functions. We think this is acceptable since Algorithm~\ref{online_op:algorithm} is suitable for a more general scenario where not only the constraints are time-varying but also each agent only knows a coordinate block of the constraint function at each round. Moreover, it does not use any parameters related to the loss and constraint functions. In contrast, the algorithms proposed in \cite{NIPS2018_7852,yuan2021distributedb} use the upper bound of the subgradients of the loss and constraint functions to design the stepsizes, and it is unclear whether the results in \cite{NIPS2018_7852,yuan2021distributedb} can still be achieved or not after extending the algorithms to suit the general scenario as considered in this paper.
	It is our ongoing work to design new distributed online algorithms such that they can achieve the same regret and cumulative constraint violation bounds as achieved by the centralized algorithm proposed in \cite{NIPS2018_7852}.
\end{remark}

Before ending this section, we would like to present some discussions on the stepsizes. Algorithm~\ref{online_op:algorithm} uses vanishing stepsizes as shown in \eqref{online_op:stepsize1} and \eqref{online_op:stepsize1_sc}, while there are some online algorithms, such as the online algorithms proposed in \cite{yi2016tracking,mokhtari2016online,zhang2019distributed,carnevale2020distributed}, used constant stepsizes. However, using vanishing stepsizes does not mean that Algorithm~\ref{online_op:algorithm} cannot be used for infinite horizons, since the results stated above hold for any time horizons, which guarantee that Algorithm~\ref{online_op:algorithm} can be used for infinite horizons. 
By the way, it should be mentioned that \cite{yi2016tracking,mokhtari2016online,zhang2019distributed,carnevale2020distributed} all assumed that the cost functions are smooth, i.e., the gradients of the cost functions are Lipschitz continuous. Such an assumption is rarely used in the papers considered  vanishing stepsizes.
Moreover, to the best of our knowledge, in the study of online convex optimization with long term constraints, there are no studies that consider non-vanishing stepsizes. We think that one possible reason for this is as follows. In the analysis in \cite{yi2016tracking,mokhtari2016online,zhang2019distributed,carnevale2020distributed}, the inequality that $f_t(x_t)-f_t(x^*_t)\ge0$ is explicitly or implicitly used. For example, that inequality has been explicitly used below equation (32) in the proof of \cite{yi2016tracking} and implicitly used to yield equation (29) in \cite{carnevale2020distributed}. However, when studying online convex optimization with long term constraints, that inequality may not hold since when choosing $x_t$ the constraints can be violated. Therefore, it is challenging to design online algorithms with non-vanishing stepsizes for online convex optimization with long term constraints and analyze their performance.

Moreover, we would like to point out that compared with related studies, the consideration of the more strict constraint violation metric is a contribution but not the key contribution of this section and does not cause significant challenges for the performance analysis either. Actually, some existing results can be extended to the scenario under the more strict constraint violation metric when using the clipped constraint functions to replace the original constraint functions.  
Instead, the key contributions of this section are (a) proposing a distributed algorithm for the general online constrained convex optimization problem which incorporates various problems studied in the literature as special cases and (b) showing the proposed algorithm achieves the same or even better performance measured by regret and the more strict constraint violation metric as explained in the above remarks, which also make the proofs more challenging. Simply speaking, the main challenge in the proofs is how to handle the error caused by the inconsistency in the local decisions at each round. It should be mentioned that due to the distributed setting the proofs are much more complicated than that for the centralized algorithms. Moreover, the proofs are different from \cite{yuan2017adaptive,yuan2021distributed,yuan2021distributedb} since we achieve strictly better results than \cite{yuan2017adaptive} as explained in Remark~\ref{online_op:remark_static}, and our algorithm is different from the algorithms in \cite{yuan2021distributed,yuan2021distributedb} even when considering the same problem settings as \cite{yuan2021distributed,yuan2021distributedb}. Similar discussions also hold for the results in the next section. In conclusion, the results in this paper can be viewed as nontrivial extensions of existing results.

\section{Distributed Online Algorithm with Bandit Feedback}\label{online_opsec:algorithm_bandit}

To handle the situations where the entire function and gradient information are not available, in this section, we focus on the bandit feedback, where at each round each agent can sample the values of its local loss and constraint functions at two points\footnote{The cases where one- and multi-point bandit feedback are available could be studied similarly, but would have different network regret and cumulative constraint violation bounds.}.

\subsection{Algorithm Description}
Under the bandit feedback setting each agent $i$ does not know the subgradients $\partial f_{i,t}(x_{i,t})$ and $\partial [g_{i,t}(x_{i,t})]_+$. Inspired by the two-point gradient estimator proposed in \cite{agarwal2010optimal,shamir2017optimal}, these subgradients can be estimated by
\begin{align*}
	\hat{\partial}f_{i,t}(x_{i,t})
	&=\frac{p}{\delta_{t}}(f_{i,t}(x_{i,t}+\delta_{t}u_{i,t})
	-f_{i,t}(x_{i,t}))u_{i,t}\in\mathbb{R}^{p},
\end{align*}
and
\begin{align*}
	\hat{\partial}[g_{i,t}(x_{i,t})]_+
	=\frac{p}{\delta_{t}}([g_{i,t}(x_{i,t}+\delta_{t}u_{i,t})]_+-[g_{i,t}(x_{i,t})]_+)^\top\otimes u_{i,t}\in\mathbb{R}^{p\times m_i},
\end{align*}
where $u_{i,t}\in\mathbb{S}^p$ is a uniformly distributed random vector, $\delta_t\in(0,r(\mathbb{X})\xi_{t}]$ is an exploration parameter, $\xi_t\in(0,1)$ is a shrinkage coefficient, and $r(\mathbb{X})$ is a known constant given in the first part in Assumption~\ref{online_op:assfunction}.

Combining our Algorithm~\ref{online_op:algorithm} with the above two-point gradient estimators, our algorithm for the bandit setting is outlined in pseudo-code as Algorithm~\ref{online_op:algorithm_bandit}.

\begin{algorithm}
	\caption{Distributed Online Algorithm with Bandit Feedback}
	\begin{algorithmic}\label{online_op:algorithm_bandit}
		\STATE \textbf{Input}:   non-increasing sequences $\{\alpha_{t}\}$, $\{\beta_{t}\}$, $\{\gamma_{t}\}\subseteq(0,+\infty)$, $\{\xi_{t}\}\subseteq(0,1)$, and $\{\delta_{t}\}\subseteq(0,r(\mathbb{X})\xi_{t}]$.
		\STATE \textbf{Initialize}:  $x_{i,1}\in(1-\xi_1)\mathbb{X}$ and $q_{i,1}={\bm 0}_{m_i}$ for all $i\in[n]$.
		\FOR{$t=1,\dots$}
		\FOR{$i=1,\dots,n$  in parallel}
		\STATE  Broadcast $x_{i,t}$ to $\mathcal{N}^{\outin}_i(\mathcal{G}_{t})$ and receive $x_{j,t}$ from $j\in\mathcal{N}^{\inout}_i(\mathcal{G}_{t})$;
		\STATE Select vector $u_{i,t}\in\mathbb{S}^{p}$ independently and uniformly at random;
		\STATE Sample $f_{i,t}(x_{i,t}+\delta_{t}u_{i,t})$, $f_{i,t}(x_{i,t})$, $g_{i,t}(x_{i,t}+\delta_{t}u_{i,t})$ and $g_{i,t}(x_{i,t})$;
		\STATE  Update \begin{align}
			z_{i,t+1}&=\sum_{j=1}^n[W_{t}]_{ij}x_{j,t},\label{online_op:al_z_bandit}\\
			\hat{\omega}_{i,t+1}&=\hat{\partial}f_{i,t}(x_{i,t})+\hat{\partial}[g_{i,t}(x_{i,t})]_+ q_{i,t},\label{online_op:al_bigomega_bandit}\\
			x_{i,t+1}&=\calP_{(1-\xi_{t+1})\mathbb{X}}(z_{i,t+1}-\alpha_{t+1}\hat{\omega}_{i,t+1}),\label{online_op:al_x_bandit}\\
			q_{i,t+1}&=[(1-\beta_{t+1}\gamma_{t+1})q_{i,t}+\gamma_{t+1}([g_{i,t}(x_{i,t})]_+
			+(\hat{\partial} [g_{i,t}(x_{i,t})]_+)^\top(x_{i,t+1}-x_{i,t}))]_{+}.\label{online_op:al_q_bandit}
		\end{align}
		\ENDFOR
		\ENDFOR
		\STATE  \textbf{Output}: $\{x_{i,t}\}$.
	\end{algorithmic}
\end{algorithm}

The sequences $\{\alpha_{t}\}$, $\{\beta_{t}\}$, $\{\gamma_{t}\}$, $\{\xi_{t}\}$, and $\{\delta_{t}\}$ used in Algorithm~\ref{online_op:algorithm_bandit} are pre-determined and the vector sequences $\{u_{i,t}\}$ are randomly selected. Moreover,
$\{z_{i,t}\}$, $\{\hat{\omega}_{i,t}\}$, $\{x_{i,t}\}$, and $\{q_{i,t}\}$ are random vector sequences generated by Algorithm~\ref{online_op:algorithm_bandit}. Let $\mathfrak{U}_t$ denote the $\sigma$-algebra generated by the independent and identically distributed random variables $u_{1,t},\dots,u_{n,t}$ and let $\mathcal{U}_t=\bigcup_{s=1}^{t}\mathfrak{U}_s$. It is straightforward to see that $\{z_{i,t+1}\}$, $\{\hat{\omega}_{i,t}\}$, $\{x_{i,t}\}$, and $\{q_{i,t}\},~i\in[n]$ depend on $\mathcal{U}_{t-1}$ and are independent of $\mathfrak{U}_s$ for all $s\ge t$.

\subsection{Performance Analysis}
This section analyzes network regret and cumulative constraint violation bounds for Algorithm~\ref{online_op:algorithm_bandit}.

Similar to the performance analysis for Algorithm~\ref{online_op:algorithm}. We have the following results.

\begin{theorem}\label{online_op:corollaryreg_bandit}
	Suppose Assumptions~\ref{online_op:assfunction}--\ref{online_op:assgraph} hold. For all $i\in[n]$, let $\{x_{i,t}\}$ be the sequences generated by Algorithm~\ref{online_op:algorithm_bandit} with
	\begin{align}\label{online_op:stepsize1_bandit}
		&\alpha_t=\frac{\alpha_0}{t^{\kappa}},~\beta_t=\frac{1}{t^{\kappa}},
		~\gamma_t=\frac{1}{t^{1-\kappa}},~\xi_{t}=\frac{1}{t+1},~\delta_{t}=\frac{r(\mathbb{X})}{t+1},~t\in\mathbb{N}_+,
	\end{align} where $\alpha_0>0$ and $\kappa\in(0,1)$. %, and $r(\mathbb{X})$ is the constant introduced in Assumption~\ref{online_op:assfunction}.
	Then, for any $T\in\mathbb{N}_+$ and any benchmark $y_{[T]}\in\calX_{T}$,
	\begin{align}
		&\mathbf{E}[\NetReg(\{x_{i,t}\},y_{[T]})]
		=\mathcal{O}\Big(\alpha_0 T^{1-\kappa}+\frac{T^\kappa(1+P_T)}{\alpha_0}\Big),\label{online_op:corollaryregequ1_bandit}\\
		&\frac{1}{n}\sum_{i=1}^n\sum_{t=1}^T\mathbf{E}[\|[g_{t}(x_{i,t})]_+\|]=
		\mathcal{O}(\sqrt{(\alpha_0+1)T^{2-\kappa}}).\label{online_op:corollaryconsequ_bandit}
	\end{align}
\end{theorem}
\begin{proof}
The explicit expressions of the right-hand sides of \eqref{online_op:corollaryregequ1_bandit}--\eqref{online_op:corollaryconsequ_bandit}, and the proof are given in  Appendix~\ref{online_op:corollaryregproof_bandit}.
\end{proof}
\begin{remark}\label{online_op:remark_dynamic_bandit}
	By setting $\alpha_0=\sqrt{1+P_T}$ and  $\kappa=0.5$ in Theorem~\ref{online_op:corollaryreg_bandit}, the dynamic regret bound achieved by the centralized online algorithm with two-point bandit feedback proposed in \cite{yi2016tracking} is recovered, although \cite{yi2016tracking} only considered the static set constraint. Moreover, in this case the dynamic regret and constraint violation bounds achieved by the centralized online algorithm with two-point bandit feedback proposed in \cite{cao2019online} are also recovered.
\end{remark}

Replacing $y_{[T]}$ with the static sequence $\check{x}^*_{[T]}$ in Theorem~\ref{online_op:corollaryreg_bandit} gives static network regret and cumulative constraint violation bounds for Algorithm~\ref{online_op:algorithm_bandit}.
\begin{corollary}\label{online_op:theoremstatic_bandit}
	Under the same conditions as stated in Theorem~\ref{online_op:corollaryreg_bandit} with $\alpha_0=1$, for any $T\in\mathbb{N}_+$, it holds that
	\begin{align}
		&\mathbf{E}[\NetReg(\{x_{i,t}\},\check{x}^*_{[T]})]
		=\mathcal{O}(T^{\max\{\kappa,1-\kappa\}}),\label{online_op:staticregequ1_bandit}\\
		&\frac{1}{n}\sum_{i=1}^n\sum_{t=1}^T\mathbf{E}[\|[g_{t}(x_{i,t})]_+\|]=
		\mathcal{O}(T^{1-\kappa/2}).
		\label{online_op:staticconsequ_bandit}
	\end{align}
\end{corollary}
\begin{remark}\label{online_op:remark_static_bandit}
	Corollary~\ref{online_op:theoremstatic_bandit} shows that the results achieved by Algorithm~\ref{online_op:algorithm_bandit} are more general than the results achieved by the online algorithms with two-point bandit feedback proposed in \cite{agarwal2010optimal,shamir2017optimal,mahdavi2012trading,yuan2021distributed}. Specifically, by setting $\kappa=0.5$ in Corollary~\ref{online_op:theoremstatic_bandit}, the results in \cite{agarwal2010optimal,shamir2017optimal,mahdavi2012trading} are recovered, although the algorithms proposed in \cite{agarwal2010optimal,shamir2017optimal,mahdavi2012trading} all are centralized, and \cite{agarwal2010optimal,shamir2017optimal} only considered the static set constraint, and \cite{mahdavi2012trading} considered static inequality constraints and full-information feedback for the loss functions. The same bounds as presented in \eqref{online_op:staticregequ1_bandit}--\eqref{online_op:staticconsequ_bandit} were also achieved by the distributed online algorithm with two-point bandit feedback proposed in \cite{yuan2021distributed} when the loss functions are quadratic and the constraint functions are time-invariant, linear, and known in advance.
\end{remark}

If Assumption~\ref{online_op:assstrongconvex} also holds, then the static network regret bound can be further reduced.
\begin{theorem}\label{online_op:corollaryreg_sc_bandit}
	Suppose Assumptions~\ref{online_op:assfunction}--\ref{online_op:assstrongconvex} hold. For all $i\in[n]$, let $\{x_{i,t}\}$ be the sequences generated by Algorithm~\ref{online_op:algorithm_bandit} with
	\begin{align}\label{online_op:stepsize1_sc_bandit}
		&\alpha_t=\frac{1}{t^{c}},~\beta_t=\frac{1}{t^{\kappa}},
		~\gamma_t=\frac{1}{t^{1-\kappa}},~\xi_{t}=\frac{1}{t+1},~\delta_{t}=\frac{r(\mathbb{X})}{t+1},~t\in\mathbb{N}_+,
	\end{align} where $c\in[\max\{\kappa,1-\kappa\},1)$ and $\kappa\in(0,1)$. Then, for any $T\in\mathbb{N}_+$, it holds that
	\begin{align}
		&\mathbf{E}[\NetReg(\{x_{i,t}\},\check{x}^*_{[T]})]
		=\mathcal{O}(T^{\kappa}),\label{online_op:corollaryregequ1_sc_bandit}\\
		&\frac{1}{n}\sum_{i=1}^n\sum_{t=1}^T\mathbf{E}[\|[g_{t}(x_{i,t})]_+\|]=
		\mathcal{O}(T^{1-\kappa/2}).
		\label{online_op:corollaryconsequ_sc_bandit}
	\end{align}
\end{theorem}
\begin{proof}
The explicit expressions of the right-hand sides of \eqref{online_op:corollaryregequ1_sc_bandit}--\eqref{online_op:corollaryconsequ_sc_bandit}, and the proof are given in  Appendix~\ref{online_op:corollaryregproof_sc_bandit}.
\end{proof}
\begin{remark}\label{online_op:remark_sc_bandit}
	By comparing Theorem~\ref{online_op:corollaryreg}, Corollary~\ref{online_op:theoremstatic}, and Theorem~\ref{online_op:corollaryreg_sc} with Theorem~\ref{online_op:corollaryreg_bandit}, Corollary~\ref{online_op:theoremstatic_bandit}, and Theorem~\ref{online_op:corollaryreg_sc_bandit}, respectively, we can see that the same network regret and cumulative constraint violation bounds are achieved under the same assumptions. In other words, in an average sense, the distributed online algorithm with two-point bandit feedback (Algorithm~\ref{online_op:algorithm_bandit}) is as efficient as the distributed online algorithm with full-information feedback (Algorithm~\ref{online_op:algorithm}).
\end{remark}

\section{SIMULATIONS}\label{online_opsec:simulation}
In this section, we evaluate the performance of Algorithms~\ref{online_op:algorithm} and \ref{online_op:algorithm_bandit} in solving the distributed online linear regression problem with time-varying linear inequality constraints.

In this problem, the local loss and constraint functions are $f_{i,t}(x)=\frac{1}{2}(H_{i,t}x-z_{i,t})^2$ and $g_{i,t}(x)=A_{i,t}x-a_{i,t}$, respectively, where $H_{i,t}\in\mathbb{R}^{d_i\times p}$, $z_{i,t}\in\mathbb{R}^{d_i}$, $A_{i,t}\in\mathbb{R}^{m_i\times p}$, and $a_{i,t}\in\mathbb{R}^{m_i}$ with $d_i\in\mathbb{N}_+$. Moreover, the constraint set is $\mathbb{X}\subseteq\mathbb{R}^{p}$. At each time $t$, an undirected random graph is used as the communication graph. Specifically, connections between agents are random and the probability of two agents being connected is $\rho$. To guarantee that Assumption~\ref{online_op:assgraph} holds, edges $(i,i+1),~i\in[n-1]$ are also added and $[W_t]_{ij}=\frac{1}{n}$ if $(j,i)\in\mathcal{E}_t$ and $[W_t]_{ii}=1-\sum_{j=1}^n[W_t]_{ij}$.

We set $n=100$, $\rho=0.1$, $d_i=4$, $p=10$, $m_i=2$, and $\mathbb{X}=[-5,5]^p$. Each component of $H_{i,t}$ is generated from the uniform distribution in the interval $[-1,1]$ and $z_{i,t}=H_{i,t}\bm{1}_p+\epsilon_{i,t}$, where $\epsilon_{i,t}$ is a standard normal random vector. Each component of $A_{i,t}$ and $a_{i,t}$ is generated from the uniform distribution in the interval $[0,2]$ and $[0,1]$, respectively.

Noting that there are no other distributed online algorithms to solve the considered problem due to the time-varying constraints, we compare our Algorithms~\ref{online_op:algorithm} and \ref{online_op:algorithm_bandit} with the centralized algorithms with full-information feedback proposed in \cite{sun2017safety,yu2017online,neely2017online}\footnote{The algorithms in \cite{yu2017online,neely2017online} are the same.} and  the centralized algorithm with two-point bandit feedback proposed in \cite{cao2019online}. Fig.~\ref{online_op:fig:reg_alg} and Fig.~\ref{online_op:fig:cons_alg} illustrate the evolutions of the average cumulative loss $\frac{1}{n}\sum_{i=1}^{n}\sum_{t=1}^{T}f_t(x_{i,t})/T$ and the average cumulative constraint violation $\frac{1}{n}\sum_{i=1}^n\sum_{t=1}^T\|[g_{t}(x_{i,t})]_+\|/T$, respectively. Fig.~\ref{online_op:fig:reg_alg} shows that the algorithms with the same kind of information feedback have almost the same average cumulative loss and the algorithms with full-information feedback have smaller average cumulative loss, which are in accordance with the theoretical results. Fig.~\ref{online_op:fig:cons_alg} shows that our proposed algorithms have  smaller average cumulative constraint violation, which also matches the theoretical results since the standard constraint violation metric rather than the more strict metric was used in \cite{sun2017safety,yu2017online,neely2017online,cao2019online}.

\begin{figure}
  \centering
  \includegraphics[width=0.99\linewidth]{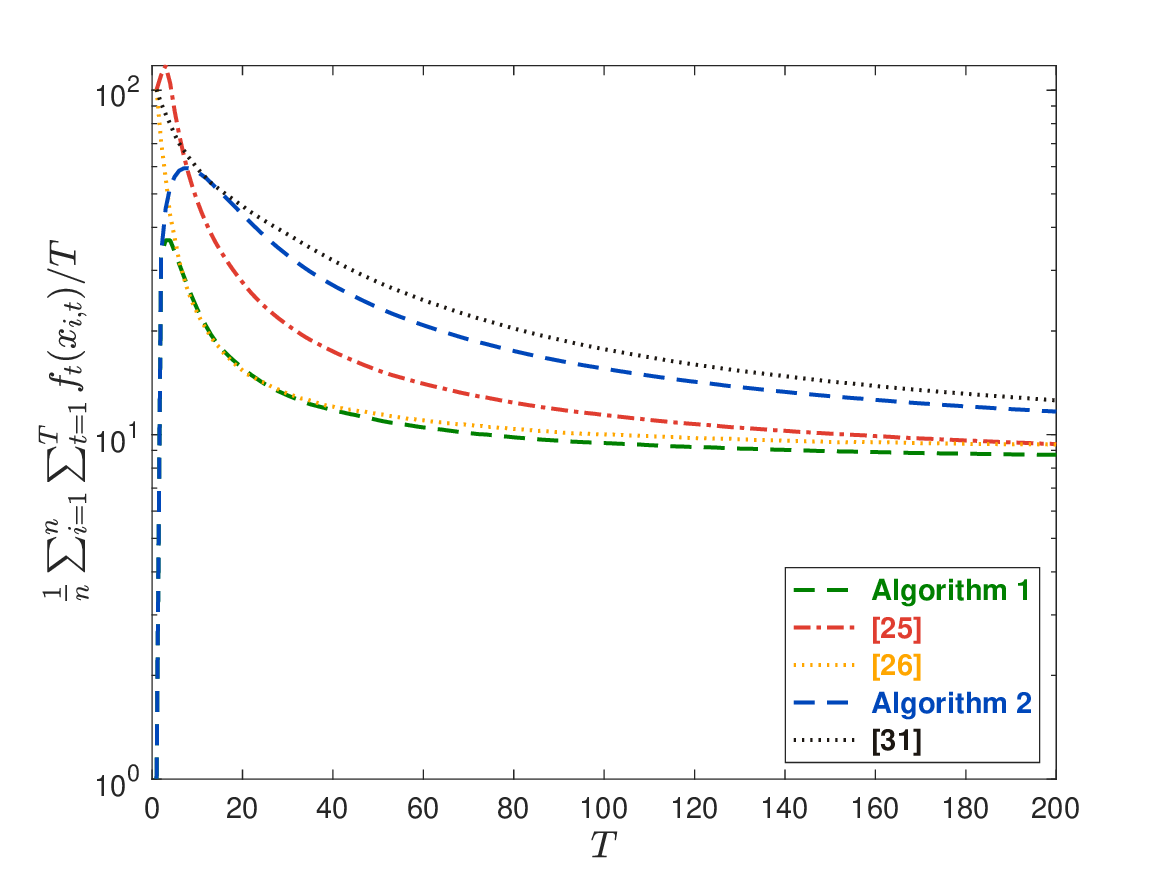}
\caption{Evolutions of $\frac{1}{n}\sum_{i=1}^{n}\sum_{t=1}^{T}f_t(x_{i,t})/T$.}
\label{online_op:fig:reg_alg}
\end{figure}

\begin{figure}
  \centering
  \includegraphics[width=0.99\linewidth]{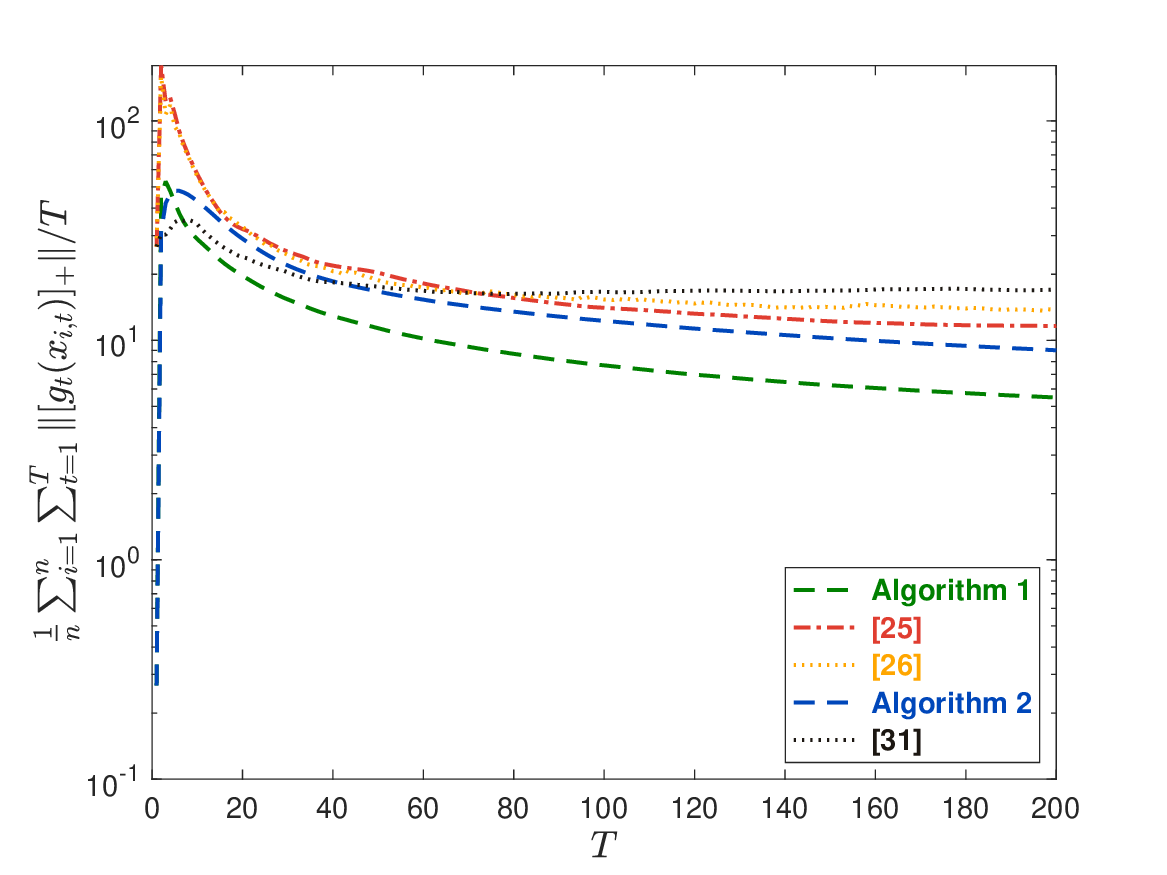}
\caption{Evolutions of $\frac{1}{n}\sum_{i=1}^n\sum_{t=1}^T\|[g_{t}(x_{i,t})]_+\|/T$.}
\label{online_op:fig:cons_alg}
\end{figure}

%We also use numerical simulations to show how the number of agents and network topology affect the performance of the proposed algorithm in the online version .

\section{CONCLUSIONS}\label{online_opsec:conclusion}
In this paper, we considered the distributed online convex optimization problem with time-varying constraints over a network of agents, which incorporates various problems studied in the literature. We proposed two distributed online algorithm to solve this problem and analyzed network regret and cumulative constraint violation bounds for the proposed algorithms under different conditions. Our results can be viewed as nontrivial extensions of existing results, where we considered  distributed and time-varying settings and used the more strict constraint violation metric. In the future, we will design new distributed online algorithms such that the static network regret bound can be further reduced under the strongly convex condition and the network cumulative constraint violation bound can be reduced when the constraint functions satisfy Slater's condition. We will also consider how to reduce communication complexity.

%\section*{ACKNOWLEDGMENTS}

\bibliographystyle{IEEEtran}
\bibliography{refs,refs_bandit}

\begin{thebibliography}{10}
\providecommand{\url}[1]{#1}
\csname url@samestyle\endcsname
\providecommand{\newblock}{\relax}
\providecommand{\bibinfo}[2]{#2}
\providecommand{\BIBentrySTDinterwordspacing}{\spaceskip=0pt\relax}
\providecommand{\BIBentryALTinterwordstretchfactor}{4}
\providecommand{\BIBentryALTinterwordspacing}{\spaceskip=\fontdimen2\font plus
\BIBentryALTinterwordstretchfactor\fontdimen3\font minus
  \fontdimen4\font\relax}
\providecommand{\BIBforeignlanguage}[2]{{%
\expandafter\ifx\csname l@#1\endcsname\relax
\typeout{** WARNING: IEEEtran.bst: No hyphenation pattern has been}%
\typeout{** loaded for the language `#1'. Using the pattern for}%
\typeout{** the default language instead.}%
\else
\language=\csname l@#1\endcsname
\fi
#2}}
\providecommand{\BIBdecl}{\relax}
\BIBdecl

\bibitem{crammer2006online}
K.~Crammer, O.~Dekel, J.~Keshet, S.~Shalev-Shwartz, and Y.~Singer, ``Online
  passive aggressive algorithms,'' \emph{Journal of Machine Learning Research},
  vol.~7, pp. 551--585, 2006.

\bibitem{mairal2009online}
J.~Mairal, F.~Bach, J.~Ponce, and G.~Sapiro, ``Online dictionary learning for
  sparse coding,'' in \emph{International Conference on Machine Learning},
  2009, pp. 689--696.

\bibitem{goldfarb2011online}
A.~Goldfarb and C.~Tucker, ``Online display advertising: Targeting and
  obtrusiveness,'' \emph{Marketing Science}, vol.~30, no.~3, pp. 389--404,
  2011.

\bibitem{cesa1996worst}
N.~Cesa-Bianchi, P.~M. Long, and M.~K. Warmuth, ``Worst-case quadratic loss
  bounds for prediction using linear functions and gradient descent,''
  \emph{IEEE Transactions on Neural Networks}, vol.~7, no.~3, pp. 604--619,
  1996.

\bibitem{gentile1999linear}
C.~Gentile and M.~K. Warmuth, ``Linear hinge loss and average margin,'' in
  \emph{Advances in Neural Information Processing Systems}, 1999, pp. 225--231.

\bibitem{gordon1999regret}
G.~J. Gordon, ``Regret bounds for prediction problems,'' in \emph{Conference on
  Learning Theory}, 1999, pp. 29--40.

\bibitem{zinkevich2003online}
M.~Zinkevich, ``Online convex programming and generalized infinitesimal
  gradient ascent,'' in \emph{International Conference on Machine Learning},
  2003, pp. 928--936.

\bibitem{hazan2007logarithmic}
E.~Hazan, A.~Agarwal, and S.~Kale, ``Logarithmic regret algorithms for online
  convex optimization,'' \emph{Machine Learning}, vol.~69, no. 2-3, pp.
  169--192, 2007.

\bibitem{agarwal2010optimal}
A.~Agarwal, O.~Dekel, and L.~Xiao, ``Optimal algorithms for online convex
  optimization with multi-point bandit feedback,'' in \emph{Conference on
  Learning Theory}, 2010, pp. 28--40.

\bibitem{shalev2012online}
S.~Shalev-Shwartz, ``Online learning and online convex optimization,''
  \emph{Foundations and Trends in Machine Learning}, vol.~4, no.~2, pp.
  107--194, 2012.

\bibitem{yi2016tracking}
T.~Yang, L.~Zhang, R.~Jin, and J.~Yi, ``Tracking slowly moving clairvoyant:
  Optimal dynamic regret of online learning with true and noisy gradient,'' in
  \emph{International Conference on Machine Learning}, 2016, pp. 449--457.

\bibitem{mokhtari2016online}
A.~Mokhtari, S.~Shahrampour, A.~Jadbabaie, and A.~Ribeiro, ``Online
  optimization in dynamic environments: Improved regret rates for strongly
  convex problems,'' in \emph{IEEE Conference on Decision and Control}, 2016,
  pp. 7195--7201.

\bibitem{hazan2016introduction}
E.~Hazan, ``Introduction to online convex optimization,'' \emph{Foundations and
  Trends in Optimization}, vol.~2, no. 3-4, pp. 157--325, 2016.

\bibitem{shamir2017optimal}
O.~Shamir, ``An optimal algorithm for bandit and zero-order convex optimization
  with two-point feedback,'' \emph{Journal of Machine Learning Research},
  vol.~18, no.~52, pp. 1--11, 2017.

\bibitem{zhang2018adaptive}
L.~Zhang, S.~Lu, and Z.-H. Zhou, ``Adaptive online learning in dynamic
  environments,'' in \emph{Advances in Neural Information Processing Systems},
  2018, pp. 1323--1333.

\bibitem{zhang2018dynamic}
L.~Zhang, T.~Yang, R.~Jin, and Z.-H. Zhou, ``Dynamic regret of strongly
  adaptive methods,'' in \emph{International Conference on Machine Learning},
  2018, pp. 5882--5891.

\bibitem{shames2019online}
I.~Shames, D.~Selvaratnam, and J.~H. Manton, ``Online optimization using zeroth
  order oracles,'' \emph{IEEE Control Systems Letters}, vol.~4, no.~1, pp.
  31--36, 2019.

\bibitem{ijcai2020-731}
L.~Zhang, ``Online learning in changing environments,'' in \emph{International
  Joint Conference on Artificial Intelligence}, 2020, pp. 5178--5182.

\bibitem{mahdavi2012trading}
M.~Mahdavi, R.~Jin, and T.~Yang, ``Trading regret for efficiency: online convex
  optimization with long term constraints,'' \emph{Journal of Machine Learning
  Research}, vol.~13, no.~81, pp. 2503--2528, 2012.

\bibitem{jenatton2016adaptive}
R.~Jenatton, J.~Huang, and C.~Archambeau, ``Adaptive algorithms for online
  convex optimization with long-term constraints,'' in \emph{International
  Conference on Machine Learning}, 2016, pp. 402--411.

\bibitem{NIPS2018_7852}
J.~Yuan and A.~Lamperski, ``Online convex optimization for cumulative
  constraints,'' in \emph{Advances in Neural Information Processing Systems},
  2018, pp. 6140--6149.

\bibitem{yu2020lowJMLR}
H.~Yu and M.~J. Neely, ``A low complexity algorithm with $ {O}(\sqrt{T})$
  regret and $ {O}(1)$ constraint violations for online convex optimization
  with long term constraints,'' \emph{Journal of Machine Learning Research},
  vol.~21, no.~1, pp. 1--24, 2020.

\bibitem{yi2021regret}
X.~Yi, X.~Li, T.~Yang, L.~Xie, T.~Chai, and K.~H. Johansson, ``Regret and
  cumulative constraint violation analysis for online convex optimization with
  long term constraints,'' in \emph{International Conference on Machine
  Learning}, 2021, pp. 11\,998--12\,008.

\bibitem{paternain2016online}
S.~Paternain and A.~Ribeiro, ``Online learning of feasible strategies in
  unknown environments,'' \emph{IEEE Transactions on Automatic Control},
  vol.~62, no.~6, pp. 2807--2822, 2016.

\bibitem{sun2017safety}
W.~Sun, D.~Dey, and A.~Kapoor, ``Safety-aware algorithms for adversarial
  contextual bandit,'' in \emph{International Conference on Machine Learning},
  2017, pp. 3280--3288.

\bibitem{yu2017online}
H.~Yu, M.~Neely, and X.~Wei, ``Online convex optimization with stochastic
  constraints,'' in \emph{Advances in Neural Information Processing Systems},
  2017, pp. 1428--1438.

\bibitem{neely2017online}
M.~J. Neely and H.~Yu, ``Online convex optimization with time-varying
  constraints,'' \emph{arXiv:1702.04783}, 2017.

\bibitem{chen2017online}
T.~Chen, Q.~Ling, and G.~B. Giannakis, ``An online convex optimization approach
  to proactive network resource allocation,'' \emph{IEEE Transactions on Signal
  Processing}, vol.~65, no.~24, pp. 6350--6364, 2017.

\bibitem{chen2018bandit}
T.~Chen and G.~B. Giannakis, ``Bandit convex optimization for scalable and
  dynamic {I}o{T} management,'' \emph{IEEE Internet of Things Journal}, vol.~6,
  no.~1, pp. 1276--1286, 2019.

\bibitem{chen2018heterogeneous}
T.~Chen, Q.~Ling, Y.~Shen, and G.~B. Giannakis, ``Heterogeneous online learning
  for ``thing-adaptive'' fog computing in {I}o{T},'' \emph{IEEE Internet of
  Things Journal}, vol.~5, no.~6, pp. 4328--4341, 2018.

\bibitem{cao2019online}
X.~Cao and K.~R. Liu, ``Online convex optimization with time-varying
  constraints and bandit feedback,'' \emph{IEEE Transactions on Automatic
  Control}, vol.~64, no.~7, pp. 2665--2680, 2019.

\bibitem{pmlr-v97-liakopoulos19a}
N.~Liakopoulos, A.~Destounis, G.~Paschos, T.~Spyropoulos, and P.~Mertikopoulos,
  ``Cautious regret minimization: Online optimization with long-term budget
  constraints,'' in \emph{International Conference on Machine Learning}, 2019,
  pp. 3944--3952.

\bibitem{sadeghi2019online}
O.~Sadeghi and M.~Fazel, ``Online continuous {DR}-submodular maximization with
  long-term budget constraints,'' in \emph{International Conference on
  Artificial Intelligence and Statistics}, 2020, pp. 4410--4419.

\bibitem{wei2020online}
X.~Wei, H.~Yu, and M.~J. Neely, ``Online primal--dual mirror descent under
  stochastic constraints,'' \emph{Proceedings of the {ACM} on Measurement and
  Analysis of Computing Systems}, vol.~4, no.~2, pp. 1--36, 2020.

\bibitem{pmlr-v97-koloskova19a}
A.~Koloskova, S.~Stich, and M.~Jaggi, ``Decentralized stochastic optimization
  and gossip algorithms with compressed communication,'' in \emph{International
  Conference on Machine Learning}, 2019, pp. 3478--3487.

\bibitem{tsianos2012distributed}
K.~I. Tsianos and M.~G. Rabbat, ``Distributed strongly convex optimization,''
  in \emph{Annual Allerton Conference on Communication, Control, and
  Computing}, 2012, pp. 593--600.

\bibitem{mateos2014distributed}
D.~Mateos-N{\'u}nez and J.~Cort{\'e}s, ``Distributed online convex optimization
  over jointly connected digraphs,'' \emph{IEEE Transactions on Network Science
  and Engineering}, vol.~1, no.~1, pp. 23--37, 2014.

\bibitem{koppel2015saddle}
A.~Koppel, F.~Y. Jakubiec, and A.~Ribeiro, ``A saddle point algorithm for
  networked online convex optimization,'' \emph{IEEE Transactions on Signal
  Processing}, vol.~63, no.~19, pp. 5149--5164, 2015.

\bibitem{hosseini2016online}
S.~Hosseini, A.~Chapman, and M.~Mesbahi, ``Online distributed convex
  optimization on dynamic networks.'' \emph{IEEE Transactions on Automatic
  Control}, vol.~61, no.~11, pp. 3545--3550, 2016.

\bibitem{akbari2017distributed}
M.~Akbari, B.~Gharesifard, and T.~Linder, ``Distributed online convex
  optimization on time-varying directed graphs,'' \emph{IEEE Transactions on
  Control of Network Systems}, vol.~4, no.~3, pp. 417--428, 2017.

\bibitem{shahrampour2018distributed}
S.~Shahrampour and A.~Jadbabaie, ``Distributed online optimization in dynamic
  environments using mirror descent,'' \emph{IEEE Transactions on Automatic
  Control}, vol.~63, no.~3, pp. 714--725, 2018.

\bibitem{akbari2018individual}
M.~Akbari, B.~Gharesifard, and T.~Linder, ``Individual regret bounds for the
  distributed online alternating direction method of multipliers,'' \emph{IEEE
  Transactions on Automatic Control}, vol.~64, no.~4, pp. 1746--1752, 2018.

\bibitem{zhang2019distributed}
Y.~Zhang, R.~J. Ravier, M.~M. Zavlanos, and V.~Tarokh, ``A distributed online
  convex optimization algorithm with improved dynamic regret,'' in \emph{IEEE
  Conference on Decision and Control}, 2019, pp. 2449--2454.

\bibitem{wan2020projection}
Y.~Wan, W.-W. Tu, and L.~Zhang, ``Projection-free distributed online convex
  optimization with $ {O}(\sqrt{T})$ communication complexity,'' in
  \emph{International Conference on Machine Learning}, 2020, pp. 9818--9828.

\bibitem{carnevale2020distributed}
G.~Carnevale, F.~Farina, I.~Notarnicola, and G.~Notarstefano, ``Distributed
  online optimization via gradient tracking with adaptive momentum,''
  \emph{arXiv:2009.01745}, 2020.

\bibitem{yuan2017adaptive}
D.~Yuan, D.~W. Ho, and G.-P. Jiang, ``An adaptive primal-dual subgradient
  algorithm for online distributed constrained optimization,'' \emph{IEEE
  Transactions on Cybernetics}, vol.~48, no.~11, pp. 3045--3055, 2017.

\bibitem{yuan2021distributed}
D.~Yuan, A.~Proutiere, and G.~Shi, ``Distributed online linear regression,''
  \emph{IEEE Transactions on Information Theory}, vol.~67, no.~1, pp. 616--639,
  2021.

\bibitem{yuan2021distributedb}
------, ``Distributed online optimization with long-term constraints,''
  \emph{IEEE Transactions on Automatic Control}, vol.~67, no.~3, pp.
  1089--1104, 2022.

\bibitem{sharma2020distributed}
P.~Sharma, P.~Khanduri, L.~Shen, D.~J. Bucci, and P.~K. Varshney, ``On
  distributed online convex optimization with sublinear dynamic regret and
  fit,'' pp. 1013--1017, 2021.

\bibitem{raginsky2011decentralized}
M.~Raginsky, N.~Kiarashi, and R.~Willett, ``Decentralized online convex
  programming with local information,'' in \emph{American Control Conference},
  2011, pp. 5363--5369.

\bibitem{lee2016coordinate}
S.~Lee, A.~Nedi{\'c}, and M.~Raginsky, ``Coordinate dual averaging for
  decentralized online optimization with nonseparable global objectives,''
  \emph{IEEE Transactions on Control of Network Systems}, vol.~5, no.~1, pp.
  34--44, 2018.

\bibitem{lee2017stochastic}
------, ``Stochastic dual averaging for decentralized online optimization on
  time-varying communication graphs,'' \emph{IEEE Transactions on Automatic
  Control}, vol.~62, no.~12, pp. 6407--6414, 2017.

\bibitem{Li2018distributed}
X.~Li, X.~Yi, and L.~Xie, ``Distributed online optimization for multi-agent
  networks with coupled inequality constraints,'' \emph{IEEE Transactions on
  Automatic Control}, vol.~66, no.~8, pp. 3575--3591, 2021.

\bibitem{yi2020distributed}
X.~Yi, X.~Li, L.~Xie, and K.~H. Johansson, ``Distributed online convex
  optimization with time-varying coupled inequality constraints,'' \emph{IEEE
  Transactions on Signal Processing}, vol.~68, pp. 731--746, 2020.

\bibitem{yi2019distributed}
X.~Yi, X.~Li, T.~Yang, L.~Xie, K.~H. Johansson, and T.~Chai, ``Distributed
  bandit online convex optimization with time-varying coupled inequality
  constraints,'' \emph{IEEE Transactions on Automatic Control}, vol.~66,
  no.~10, pp. 4620--4635, 2021.

\bibitem{li2020distributed}
X.~Li, X.~Yi, and L.~Xie, ``Distributed online convex optimization with an
  aggregative variable,'' \emph{IEEE Transactions on Control of Network
  Systems}, vol.~9, no.~1, pp. 438--449, 2022.

\bibitem{nedic2009distributed}
A.~Nedi{\'c} and A.~Ozdaglar, ``Distributed subgradient methods for multi-agent
  optimization,'' \emph{IEEE Transactions on Automatic Control}, vol.~54,
  no.~1, pp. 48--61, 2009.

\bibitem{nedic2015distributed}
A.~Nedi{\'c} and A.~Olshevsky, ``Distributed optimization over time-varying
  directed graphs,'' \emph{IEEE Transactions on Automatic Control}, vol.~60,
  no.~3, pp. 601--615, 2015.

\end{thebibliography}

\appendix\label{online_op:appendix}
\subsection{Useful Lemmas}\label{online_op:app-lemmas}

The following results are used in the proofs.

%We first state two results on the consensus protocol with perturbations and time-varying communication graphs.
\begin{lemma}\label{online_op:lemma_stoc}
(\cite{nedic2009distributed,nedic2015distributed})
Let $W_{t}$ be the  adjacency matrix associated with a time-varying graph satisfying Assumption~\ref{online_op:assgraph}.
Then,
\begin{align}
&\Big|[\Psi_{s}^{t}]_{ij}-\frac{1}{n}\Big|\le \tau\lambda^{t-s},~\forall i,j\in[n],~\forall t\ge s\ge1,\label{online_op:lemma_stoceqm}
\end{align}
where $\Psi_{s}^{t}=W_tW_{t-1}\cdots W_s$,  $\tau=(1-w/4n^2)^{-2}>1,$ and $\lambda=(1-w/4n^2)^{1/B}\in(0,1)$.
\end{lemma}
%The proof of this lemma can be found in \cite{nedic2009distributed,nedic2015distributed}.

\begin{lemma}\label{online_op:lemma_projection}
(Lemma~3  in \cite{yi2019distributed})
Let $\mathbb{K}$ be a nonempty closed convex subset of $\mathbb{R}^{p}$ and let $a,~b,~c$ be three vectors in $\mathbb{R}^{p}$. The following statements hold.
\begin{enumerate}[label=(\alph*)]
%\item For each $x\in\mathbb{R}^p$, $\calP_{\mathbb{K}}(x)$  exists and is unique.

%\item $\calP_{\mathbb{K}}(x)$ is nonexpansive, i.e.,
%\begin{align}\label{dbco:lemma:projection:none}
%\left\|\calP_{\mathbb{K}}(x)-\calP_{\mathbb{K}}(y)\right\|\le\|x-y\|,~\forall x,y\in\mathbb{R}^p.
%\end{align}
\item If $a\le b$, then
\begin{align}
\|[a]_+\|\le\|b\|~\text{and}~[a]_+\le[b]_+.\label{dbco:lemma:projection:ab}
\end{align}

\item If $x_1=\calP_{\mathbb{K}}(c-a)$, then
\begin{align}
&2\langle x_1-y,a \rangle
\le\|y-c\|^2-\|y-x_1\|^2-\|x_1-c\|^2,~\forall y\in\mathbb{K}.\label{dbco:lemma:projection:xy}
\end{align}
\end{enumerate}
\end{lemma}

\begin{lemma}\label{dbco:lemma:uniformsmoothing}
Let $f:\mathbb{K}\rightarrow\mathbb{R}^m$ be a vector-valued function with $\mathbb{K}\subset\mathbb{R}^p$ being a convex and closed set. Moreover, there exists  $r(\mathbb{K})>0$ such that $r(\mathbb{K})\mathbb{B}^p\subseteq\mathbb{K}$. Denote
\begin{align*}
\hat{\partial} f(x)&=\frac{p}{\delta}(f(x+\delta u)-f(x))^\top\otimes u,~\forall x\in(1-\xi)\mathbb{K},\\
\hat{f}(x)&=\mathbf{E}_{v\in\mathbb{B}^p}[f(x+\delta v)],~\forall x\in(1-\xi)\mathbb{K},
\end{align*}
where $u\in\mathbb{S}^p$ is a uniformly distributed random vector, $\delta\in(0,r(\mathbb{K})\xi]$, $\xi\in(0,1)$, and the expectation is taken with respect to uniform distribution. The following statements hold.
\begin{enumerate}[label=(\alph*)]
\item The function $\hat{f}$ is differentiable on $(1-\xi)\mathbb{K}$ and
\begin{align*}
\partial \hat{f}(x)
=\mathbf{E}_{u\in\mathbb{S}^p}[\hat{\partial}f(x)],~\forall x\in(1-\xi)\mathbb{K}.
\end{align*}

\item If $f$ is convex on $\mathbb{K}$, then $\hat{f}$ is convex on $(1-\xi)\mathbb{K}$ and
\begin{align*}
f(x)\le \hat{f}(x),~\forall x\in(1-\xi)\mathbb{K}.
\end{align*}

\item If $f$ is Lipschitz-continuous on $\mathbb{K}$ with constant $L_0(f)>0$, then $\hat{f}$ is Lipschitz-continuous on $(1-\xi)\mathbb{K}$ with constants $L_0(f)$. Moreover, for all $x\in(1-\xi)\mathbb{K}$,
\begin{align*}
\|\hat{f}(x)-f(x)\|\le\delta L_0(f),~\|\hat{\partial}f(x)\|\le pL_0(f).
\end{align*}

\item If $f$ is bounded on $\mathbb{K}$, i.e., there exists $F_0(f)>0$ such that $\|f(x)\|\le F_0(f),~\forall x\in\mathbb{K}$, then
\begin{align*}
\|\hat{f}(x)\|\le F_0(f),~\forall x\in(1-\xi)\mathbb{K}.
\end{align*}

\item If $f$ is strongly convex with constant $\mu>0$ over $\mathbb{K}$, then $\hat{f}$ is strongly convex with constant $\mu>0$ over $(1-\xi)\mathbb{K}$.
\end{enumerate}
\end{lemma}
\begin{proof}
The proofs of the first four parts follow the proof of Lemma~2 in \cite{yi2019distributed} since they extend Lemma~2 in \cite{yi2019distributed} from scalar-valued functions to vector-valued functions.

For any $x,y\in(1-\xi)\mathbb{K}$, we have
\begin{align*}
\hat{f}(x)-\hat{f}(y)&=\mathbf{E}_{v\in\mathbb{B}^p}[f(x+\delta v)-f(y+\delta v)]\\
&\ge\mathbf{E}_{v\in\mathbb{B}^p}\Big[\langle x-y,\partial f(y+\delta v)\rangle+\frac{\mu}{2}\|x-y\|^2\Big]\\
&=\langle x-y,\mathbf{E}_{v\in\mathbb{B}^p}[\partial f(y+\delta v)]\rangle+\frac{\mu}{2}\|x-y\|^2\\
&=\langle x-y,\partial \hat{f}(y)\rangle+\frac{\mu}{2}\|x-y\|^2.
\end{align*}
Thus, the last part holds.
\end{proof}

\subsection{Proof of Theorem~\ref{online_op:corollaryreg}}\label{online_op:corollaryregproof}
Denote $\bar{x}_{t}=\frac{1}{n}\sum_{i=1}^nx_{i,t}$, $\epsilon^x_{i,t-1}=x_{i,t}-z_{i,t}$, $\Delta_{i,t}(\mu_i)=\frac{1}{2\gamma_t}(\|q_{i,t}-\mu_i\|^2
-(1-\beta_t\gamma_t)\|q_{i,t-1}-\mu_i\|^2)$ with $\mu_i$ being an arbitrary vector in $\mathbb{R}_+^{m_i}$, $b_{i,t}=[g_{i,t-1}(x_{i,t-1})]_+ +(\partial [g_{i,t-1}(x_{i,t-1})]_+)^\top(x_{i,t}-x_{i,t-1})$, $\varepsilon_{1}=2(F_1+F_2R(\mathbb{X}))^2$, $\varepsilon_2=\frac{\tau}{\lambda(1-\lambda)}\sum_{i=1}^n\|x_{i,1}\|$, $\varepsilon_3=2F_2+\frac{n^2F_2\tau^2}{2(1-\lambda)^2}$, $\varepsilon_4=2F_2\varepsilon_3+\frac{F_2^2}{4}$, $\varepsilon_5=2F_2\varepsilon_3+\varepsilon_4$, $\varepsilon_6=40\varepsilon_5$,  and $\mu_{ij}^0=\frac{\sum_{t=1}^T[g_{i,t}(x_{j,t})]_+}
{\frac{1}{\gamma_1}
+\sum_{t=1}^T(\beta_{t}+2\varepsilon_6\alpha_{t})}$.

To prove Theorem~\ref{online_op:corollaryreg}, we need some preliminary results.
%\subsection{Proof of Lemma~\ref{online_op:lemma_neterror}}\label{online_op:lemma_neterrorproof}
Firstly, we quantify the disagreement among agents. %A similar result has been shown in \cite{shahrampour2018distributed}.
\begin{lemma}\label{online_op:lemma_neterror}
If Assumption~\ref{online_op:assgraph} holds. For all $i\in[n]$ and $t\in\mathbb{N}_+$, $x_{i,t}$ generated by Algorithm~\ref{online_op:algorithm} satisfy
\begin{align}
\|x_{i,t}-\bar{x}_{t}\|
&\le \tau \lambda^{t-2}\sum_{j=1}^n\|x_{j,1}\|
+\frac{1}{n}\sum_{j=1}^n\|\epsilon^x_{j,t-1}\|
+\|\epsilon^x_{i,t-1}\|+\tau\sum_{s=1}^{t-2}\lambda^{t-s-2}\sum_{j=1}^n\|\epsilon^x_{j,s}\|.
\label{online_op:xitxbar}
\end{align}
\end{lemma}
\begin{proof}
%See Appendix~\ref{online_op:lemma_neterrorproof}.

\noindent Noting that  $\epsilon^x_{i,t-1}=x_{i,t}-z_{i,t}$, we can rewrite (\ref{online_op:al_z}) as
\begin{align*}
x_{i,t}=\sum_{j=1}^n[W_{t-1}]_{ij}x_{j,t-1}+\epsilon^x_{i,t-1}.
\end{align*}
Hence,
\begin{align*}
x_{i,t}=\sum_{j=1}^n[\Psi_{1}^{t-1}]_{ij}x_{j,1}+\epsilon^x_{i,t-1}
+\sum_{s=1}^{t-2}\sum_{j=1}^n[\Psi_{s+1}^{t-1}]_{ij}\epsilon^x_{j,s},
\end{align*}
where $\Psi_{s}^t$ is defined in Lemma \ref{online_op:lemma_stoc}.
Noting that $W_t$ is doubly stochastic for all $t\in\mathbb{N}_+$, from the above equality we have
\begin{align*}
\bar{x}_{t}=\frac{1}{n}\sum_{j=1}^nx_{j,1}
+\frac{1}{n}\sum_{j=1}^n\epsilon^x_{j,t-1}+\frac{1}{n}\sum_{s=1}^{t-2}\sum_{j=1}^n\epsilon^x_{j,s}.
\end{align*}
Thus,
\begin{align*}
x_{i,t}-\bar{x}_{t}
=\sum_{j=1}^n\Big([\Psi_{1}^{t-1}]_{ij}-\frac{1}{n}\Big)x_{j,1}
+\frac{1}{n}\sum_{j=1}^n(\epsilon^x_{i,t-1}-\epsilon^x_{j,t-1})
+\sum_{s=1}^{t-2}\sum_{j=1}^n\Big([\Psi_{s+1}^{t-1}]_{ij}-\frac{1}{n}\Big)
\epsilon^x_{j,s},
\end{align*}
which further implies
\begin{align*}
\|x_{i,t}-\bar{x}_{t}\|
&\le\sum_{j=1}^n\Big|[\Psi_{1}^{t-1}]_{ij}-\frac{1}{n}\Big|\|x_{j,1}\|
+\frac{1}{n}\sum_{j=1}^n(\|\epsilon^x_{i,t-1}\|+\|\epsilon^x_{j,t-1}\|)
+\sum_{s=1}^{t-2}\sum_{j=1}^n\Big|[\Psi_{s+1}^{t-1}]_{ij}-\frac{1}{n}\Big|
\|\epsilon^x_{j,s}\|.
\end{align*}
Hence, from (\ref{online_op:lemma_stoceqm}) and the above inequality, we have \eqref{online_op:xitxbar}.
\end{proof}

Then, we present a result on the evolution of local dual variables, which is critical to analyze the performance of Algorithm~\ref{online_op:algorithm}.
\begin{lemma}\label{online_op:lemma_virtualbound}
Suppose Assumptions~\ref{online_op:assfunction}--\ref{online_op:assgraph} hold and $\gamma_{t}\beta_{t}\le1,~t\in\mathbb{N}_+$. For all $i\in[n]$ and $t\in\mathbb{N}_+$, the sequences $q_{i,t}$ generated by Algorithm~\ref{online_op:algorithm} satisfy
\begin{align}
\Delta_{i,t}(\mu_i)&\le \varepsilon_{1}\gamma_t
+q_{i,t-1}^\top b_{i,t}-\mu_i^\top[g_{i,t-1}(x_{i,t-1})]_+ 
+\frac{1}{2}\beta_t\|\mu_i\|^2+F_2\|\mu_i\|\|x_{i,t}-x_{i,t-1}\|.
\label{online_op:gvirtualnorm}
\end{align}
\end{lemma}
\begin{proof}
%See Appendix~\ref{online_op:lemma_virtualboundproof}.

We first use mathematical induction to prove
\begin{align}
\|\beta_tq_{i,t}\|&\le F_1.\label{online_op:lemma_virtualboundeqy}
\end{align}

It is straightforward to see that $\|\beta_1q_{i,1}\|\le F_1,~\forall i\in[n]$ since $q_{i,1}={\bm 0}_{m_i},~\forall i\in[n]$. Assume now that it is true at time slot $t$ for all $i\in[n]$, i.e., $\|\beta_tq_{i,t}\|\le F_1$. We show that it remains true at time slot $t+1$.

Noting that $[g_{i,t}]_+$ is convex since $g_{i,t}$ is convex and that $\partial [g_{i,t}(x_{i,t})]_+$ is the subgradient of $[g_{i,t}]_+$ at $x_{i,t}$, we have
\begin{align}
b_{i,t+1}\le[g_{i,t}(x_{i,t+1})]_+.\label{online_op:qg}
\end{align}
Then, we have
\begin{align*}
\|q_{i,t+1}\|&=\|[(1-\beta_{t+1}\gamma_{t+1})q_{i,t}+\gamma_{t+1}b_{i,t+1}]_{+}\|\\
&\le\|(1-\gamma_{t+1}\beta_{t+1})q_{i,t}+\gamma_{t+1}[g_{i,t}(x_{i,t+1})]_+\|\\
&\le(1-\gamma_{t+1}\beta_{t+1})\|q_{i,t}\|
+\gamma_{t+1}\|g_{i,t}(x_{i,t+1})\|\\
&\le(1-\gamma_{t+1}\beta_{t+1})\frac{ F_1}{\beta_{t}}+\gamma_{t+1} F_1\\
&\le(1-\gamma_{t+1}\beta_{t+1})\frac{ F_1}{\beta_{t+1}}+\gamma_{t+1} F_1= \frac{ F_1}{\beta_{t+1}},~\forall i\in[n],
\end{align*}
where the first equality holds due to \eqref{online_op:al_q}; the first inequality holds due to \eqref{dbco:lemma:projection:ab} and (\ref{online_op:qg}); the third inequality holds due to $\|\beta_tq_{i,t}\|\le F_1$ and \eqref{online_op:ftgtupper-b}; and the last inequality holds since the sequence $\{\beta_t\}$ is non-increasing and $\gamma_{t}\beta_{t}\le1$.
Thus, the result follows.

We then prove \eqref{online_op:gvirtualnorm}.

For any $\mu_i\in\mathbb{R}^{m_i}_+$, from that the projection $[\cdot]_+$ is nonexpansive and (\ref{online_op:al_q}), we have
\begin{align}
\|q_{i,t}-\mu_i\|^2
&=\|[(1-\beta_t\gamma_t)q_{i,t-1}+\gamma_tb_{i,t}]_+-[\mu_i]_+\|^2\nonumber\\
&\le\|(1-\beta_t\gamma_t)q_{i,t-1}+\gamma_tb_{i,t}-\mu_i\|^2\nonumber\\
&=\|q_{i,t-1}-\mu_i\|^2+\gamma_t^2\|b_{i,t}-\beta_tq_{i,t-1}\|^2
+2\gamma_tq_{i,t-1}^\top b_{i,t}-2\gamma_t\mu_i^\top [g_{i,t-1}(x_{i,t-1})]_+\nonumber\\
&\quad
-2\beta_t\gamma_t(q_{i,t-1}-\mu_i)^\top q_{i,t-1}
-2\gamma_t\mu_i^\top(\partial [g_{i,t-1}(x_{i,t-1})]_+)^\top(x_{i,t}-x_{i,t-1}).
\label{online_op:gvirtualnorm_pf}
\end{align}

We have
\begin{align}\label{online_op:gvirtualnorm_pf2}
\|b_{i,t}-\beta_tq_{i,t-1}\|&\le \|b_{i,t}\|+\|\beta_tq_{i,t-1}\|\nonumber\\
&\le\|[g_{i,t-1}(x_{i,t-1})]_+ +(\partial [g_{i,t-1}(x_{i,t-1})]_+)^\top(x_{i,t}-x_{i,t-1})\|+\beta_t\frac{F_1}{\beta_{t-1}}\nonumber\\
&\le\|[g_{i,t-1}(x_{i,t-1})]_+\|+\|\partial [g_{i,t-1}(x_{i,t-1})]_+\|\|x_{i,t}-x_{i,t-1}\|+F_1\nonumber\\
&\le\|g_{i,t-1}(x_{i,t-1})\|
+\|\partial g_{i,t-1}(x_{i,t-1})\|\|x_{i,t}-x_{i,t-1}\|+F_1\nonumber\\
&\le 2F_1+2F_2R(\mathbb{X}),
\end{align}
where the second inequality holds due to \eqref{online_op:lemma_virtualboundeqy}; the third inequality holds since $\{\beta_t\}$ is a non-increasing sequence; and the last inequality holds due to \eqref{online_op:domainupper}, (\ref{online_op:ftgtupper-b}), and (\ref{online_op:subgupper-b}).

We have
\begin{align}
-2\beta_t\gamma_t(q_{i,t-1}-\mu_i)^\top q_{i,t-1}
\le\beta_t\gamma_t(\|\mu_i\|^2-\|q_{i,t-1}-\mu_i\|^2).\label{online_op:qmu3}
\end{align}

From (\ref{online_op:subgupper-b}), we have
\begin{align}
&-2\gamma_t\mu_i^\top(\partial [g_{i,t-1}(x_{i,t-1})]_+)^\top(x_{i,t}-x_{i,t-1})\le2\gamma_tF_2\|\mu_i\|\|x_{i,t}-x_{i,t-1}\|.\label{online_op:qmu2}
\end{align}

Finally, from \eqref{online_op:gvirtualnorm_pf}--\eqref{online_op:qmu2}, we have \eqref{online_op:gvirtualnorm}.
\end{proof}

%\subsection{Proof of Lemma~\ref{online_op:lemma_regretdelta}}\label{online_op:lemma_regretdeltaproof}
Next,  we provide  network regret bound at one slot.
\begin{lemma}\label{online_op:lemma_regretdelta}
Suppose Assumptions~\ref{online_op:assfunction}--\ref{online_op:assgraph} hold. For all $i\in[n]$, let $\{x_{i,t}\}$ be the sequences generated by Algorithm \ref{online_op:algorithm} and $\{y_{t}\}$ be an arbitrary sequence in $\mathbb{X}$, then
\begin{align}\label{online_op:lemma_regretdeltaequ}
\frac{1}{n}\sum_{i=1}^nf_{t}(x_{i,t})
-f_{t}(y_t)
&\le \frac{1}{n}\sum_{i=1}^nq_{i,t}^\top ([g_{i,t}(y_{t})]_+-b_{i,t+1})
-\frac{1}{n}\sum_{i=1}^n\frac{1}{2\alpha_{t+1}}\|\epsilon^x_{i,t}\|^2\nonumber\\
&\quad+\frac{1}{n}\sum_{i=1}^nF_2(2\|x_{i,t}-\bar{x}_{i,t}\|+\|x_{i,t}-x_{i,t+1}\|)\nonumber\\
&\quad+\frac{1}{ 2n\alpha_{t+1}}\sum_{i=1}^n(\|y_t-z_{i,t+1}\|^2-\|y_{t+1}-z_{i,t+2}\|^2\nonumber\\
&\quad+\|y_{t+1}-x_{i,t+1}\|^2-\|y_t-x_{i,t+1}\|^2).
\end{align}
\end{lemma}
\begin{proof}
%See Appendix~\ref{online_op:lemma_regretdeltaproof}.

From the third item in Assumption~\ref{online_op:assfunction} and Lemma~2.6 in \cite{shalev2012online}, it follows that for all $i\in[n],~t\in\mathbb{N}_+,~x,y\in \mathbb{X}$,
\begin{subequations}
\begin{align}
&\left|f_{i,t}(x)-f_{i,t}(y)\right|\le F_2\|x-y\|,\label{online_op:assfunction:functionLipf}\\
      &\|g_{i,t}(x)-g_{i,t}(y)\|\le F_2\|x-y\|.\label{online_op:assfunction:functionLipg}
\end{align}
\end{subequations}

We have
\begin{align}\label{online_op:lxit}
\frac{1}{n}\sum_{i=1}^{n}f_t(x_{i,t})
&=\frac{1}{n}\sum_{i=1}^{n}\Big(\frac{1}{n}\sum_{j=1}^{n}f_{j,t}(x_{i,t})\Big)\nonumber\\
&=\frac{1}{n}\sum_{i=1}^{n}\Big(\frac{1}{n}\sum_{j=1}^{n}f_{j,t}(x_{j,t})
+\frac{1}{n}\sum_{j=1}^{n}
(f_{j,t}(x_{i,t})-f_{j,t}(x_{j,t}))\Big)\nonumber\\
&=\frac{1}{n}\sum_{i=1}^{n}f_{i,t}(x_{i,t})
+\frac{1}{n^2}\sum_{i=1}^{n}\sum_{j=1}^{n}
(f_{j,t}(x_{i,t})-f_{j,t}(x_{j,t}))\nonumber\\
&\le\frac{1}{n}\sum_{i=1}^{n}f_{i,t}(x_{i,t})
+\frac{1}{n^2}\sum_{i=1}^{n}\sum_{j=1}^{n}
G\|x_{i,t}-x_{j,t}\|\nonumber\\
&\le\frac{1}{n}\sum_{i=1}^{n}f_{i,t}(x_{i,t})+\frac{2G}{n}\sum_{i=1}^{n}
\|x_{i,t}-\bar{x}_{t}\|,
\end{align}
where the first inequality holds due to \eqref{online_op:assfunction:functionLipf}.

Noting that $f_{i,t}$ is convex, from \eqref{online_op:subgupper-a}, we have
\begin{align}\label{online_op:fxy}
f_{i,t}(x_{i,t})-f_{i,t}(y_t)&\le\langle\partial f_{i,t}(x_{i,t}),x_{i,t}-y_t\rangle\nonumber\\
&=\langle\partial f_{i,t}(x_{i,t}),x_{i,t}-x_{i,t+1}\rangle
+\langle\partial f_{i,t}(x_{i,t}),x_{i,t+1}-y_t\rangle\nonumber\\
&\le F_2\|x_{i,t}-x_{i,t+1}\|+\langle\partial f_{i,t}(x_{i,t}),x_{i,t+1}-y_t\rangle.
\end{align}
For the second term of \eqref{online_op:fxy}, from \eqref{online_op:al_bigomega}, we have
\begin{align}
\langle\partial f_{i,t}(x_{i,t}),x_{i,t+1}-y_t\rangle
&=\langle\omega_{i,t+1},x_{i,t+1}-y_t\rangle+\langle\partial [g_{i,t}(x_{i,t})]_+ q_{i,t},y_t-x_{i,t+1}\rangle
\nonumber\\
&=\langle\omega_{i,t+1},x_{i,t+1}-y_t\rangle+\langle\partial [g_{i,t}(x_{i,t})]_+ q_{i,t},y_t-x_{i,t}\rangle\nonumber\\
&\quad+\langle\partial [g_{i,t}(x_{i,t})]_+ q_{i,t},x_{i,t}-x_{i,t+1}\rangle.\label{online_op:fxy1}
\end{align}
Noting that $\partial [g_{i,t}(x_{i,t})]_+$ is the subgradient of the convex function $[g_{i,t}]_+$ at $x_{i,t}$, from $q_{i,t}\ge{\bm 0}_{m_i},~\forall t\in\mathbb{N}_+,~\forall i\in[n]$, we have
\begin{align}
&\langle\partial [g_{i,t}(x_{i,t})]_+ q_{i,t},y_t-x_{i,t}\rangle\le q_{i,t}^\top [g_{i,t}(y_{t})]_+ -q_{i,t}^\top [g_{i,t}(x_{i,t})]_+
.\label{online_op:gyxdelta}
\end{align}
Applying \eqref{dbco:lemma:projection:xy} to the update (\ref{online_op:al_x}), we get
\begin{align}
\langle\omega_{i,t+1},x_{i,t+1}-y_t\rangle
&\le\frac{1}{2\alpha_{t+1}}(\|y_t-z_{i,t+1}\|^2-\|y_t-x_{i,t+1}\|^2
-\|\epsilon^x_{i,t}\|^2)\nonumber\\
&=\frac{1}{2\alpha_{t+1}}(\|y_t-z_{i,t+1}\|^2-\|y_{t+1}-z_{i,t+2}\|^2\nonumber\\
&\quad+\|y_{t+1}-z_{i,t+2}\|^2-\|y_t-x_{i,t+1}\|^2-\|\epsilon^x_{i,t}\|^2)\nonumber\\
&=\frac{1}{2\alpha_{t+1}}\Big(\|y_t-z_{i,t+1}\|^2-\|y_{t+1}-z_{i,t+2}\|^2-\|\epsilon^x_{i,t}\|^2\nonumber\\
&\quad+\Big\|y_{t+1}-\sum_{j=1}^n[W_{t+1}]_{ij}x_{j,t+1}\Big\|^2-\|y_t-x_{i,t+1}\|^2
\Big)\nonumber\\
&\le\frac{1}{2\alpha_{t+1}}\Big(\|y_t-z_{i,t+1}\|^2-\|y_{t+1}-z_{i,t+2}\|^2-\|\epsilon^x_{i,t}\|^2\nonumber\\
&\quad+\sum_{j=1}^n[W_{t+1}]_{ij}\|y_{t+1}-x_{j,t+1}\|^2-\|y_t-x_{i,t+1}\|^2
\Big),\label{online_op:omgea2}
\end{align}
where the last inequality holds since $W_{t+1}$ is doubly stochastic and $\|\cdot\|^2$ is convex.

Combining (\ref{online_op:fxy})--(\ref{online_op:omgea2}), summing over $i\in[n]$, and dividing by $n$, and using $\sum_{i=1}^n[W_t]_{ij}=1,~\forall t\in\mathbb{N}_+$ yields
 (\ref{online_op:lemma_regretdeltaequ}).
\end{proof}

Finally,  we show network regret and cumulative constraint violation bounds.
\begin{lemma}\label{online_op:theoremreg}
Suppose Assumptions~\ref{online_op:assfunction}--\ref{online_op:assgraph} hold and $\gamma_{t}\beta_{t}\le1,~t\in\mathbb{N}_+$. For all $i\in[n]$, let $\{x_{i,t}\}$ be the sequences generated by Algorithm~\ref{online_op:algorithm}. Then, for any comparator sequence $y_{[T]}\in\calX_{T}$,
\begin{align}
\NetReg(\{x_{i,t}\},y_{[T]})
&\le 4F_2\varepsilon_2
+\sum_{t=1}^T(\varepsilon_1\gamma_{t}
+10\varepsilon_5\alpha_{t})+\frac{2R(\mathbb{X})^2}{\alpha_{T+1}}
+\frac{2R(\mathbb{X})}{\alpha_{T}}P_T\nonumber\\
&\quad
-\frac{1}{2n}\sum_{t=1}^T\sum_{i=1}^n\Big(\frac{1}{\gamma_{t}}
-\frac{1}{\gamma_{t+1}}+\beta_{t+1}\Big)\|q_{i,t}\|^2,\label{online_op:theoremregequ}\\
\frac{1}{n}\sum_{i=1}^n\Big\|\sum_{t=1}^T[g_{t}(x_{i,t})]_+\Big\|^2
&\le 4n\varepsilon_2F_1F_2T+2\Big(\frac{1}{\gamma_1}
+\sum_{t=1}^T(\beta_{t}+\varepsilon_6\alpha_{t})\Big)
\Big(nF_1T\nonumber\\
&\quad
+\sum_{t=1}^Tn(\varepsilon_1\gamma_{t}
+20\varepsilon_5\alpha_{t})+\frac{2nR(\mathbb{X})^2}{\alpha_{T+1}}\nonumber\\
&\quad
-\frac{1}{2}\sum_{t=1}^T\sum_{i=1}^n\Big(\frac{1}{\gamma_{t}}
-\frac{1}{\gamma_{t+1}}+\beta_{t+1}\Big)\|q_{i,t}-\mu_{ij}^0\|^2\Big).
\label{online_op:theoremconsequ}
\end{align}
\end{lemma}
\begin{proof}
\noindent {\bf (i)} We first provide a loose bound for network regret.

From (\ref{online_op:domainupper}), we have
\begin{align}\label{online_op:dxy}
\|y_{t+1}-x_{i,t+1}\|^2-\|y_t-x_{i,t+1}\|^2
&\le\|y_{t+1}-y_t\|\|y_{t+1}-x_{i,t+1}+y_t-x_{i,t+1}\|\nonumber\\
&\le4R(\mathbb{X})\|y_{t+1}-y_t\|.
\end{align}

From \eqref{online_op:gvirtualnorm}, (\ref{online_op:lemma_regretdeltaequ}), and \eqref{online_op:dxy},  and noting that $g_{i,t}(y_t)\le{\bm 0}_{m_i},~\forall i\in[n]$ when $y_{[T]}\in\calX_{T}$, we have
\begin{align}\label{online_op:theoremregequ1}
&\frac{1}{ n}\sum_{i=1}^n\Big(\Delta_{i,t+1}(\mu_i)+\mu_i^\top[g_{i,t}(x_{i,t})]_+
-\frac{1}{2}\beta_{t+1}\|\mu_i\|^2\Big)+\frac{1}{n}\sum_{i=1}^nf_{t}(x_{i,t})
-f_{t}(y_t)\nonumber\\
&\le \varepsilon_1\gamma_{t+1}+\frac{1}{n}\sum_{i=1}^n
\tilde{\Delta}_{i,t+1}(\mu_i)+\frac{2R(\mathbb{X})}{\alpha_{t+1}}\|y_{t+1}-y_t\|\nonumber\\
&\quad+\frac{1}{ 2n\alpha_{t+1}}\sum_{i=1}^n(\|y_t-z_{i,t+1}\|^2-\|y_{t+1}-z_{i,t+2}\|^2),
\end{align}
where
\begin{align*}
\tilde{\Delta}_{i,t+1}(\mu_i)&=F_2(\|\mu_i\|+1)\|x_{i,t}-x_{i,t+1}\|
+2F_2\|x_{i,t}-\bar{x}_{i,t}\|-\frac{1}{2\alpha_{t+1}}\|\epsilon^x_{i,t}\|^2.
\end{align*}

From $\{\alpha_t\}$ is non-increasing and (\ref{online_op:domainupper}), we have
\begin{align}
&\sum_{t=1}^T\frac{1}{\alpha_{t+1}}(\|y_t-z_{i,t+1}\|^2-\|y_{t+1}-z_{i,t+2}\|^2)\nonumber\\
&=\sum_{t=1}^T\Big(\frac{1}{\alpha_{t}}\|y_t-z_{i,t+1}\|^2-\frac{1}{\alpha_{t+1}}
\|y_{t+1}-z_{i,t+2}\|^2
+\Big(\frac{1}{\alpha_{t+1}}-\frac{1}{\alpha_{t}}\Big)\|y_t-z_{i,t+1}\|^2\Big)\nonumber\\
&\le\frac{1}{\alpha_{1}}\|y_1-z_{i,2}\|^2-\frac{1}{\alpha_{T+1}}\|y_{T+1}-z_{i,T+2}\|^2
+\sum_{t=1}^T\Big(\frac{1}{\alpha_{t+1}}-\frac{1}{\alpha_{t}}\Big)4R(\mathbb{X})^2
\le\frac{4R(\mathbb{X})^2}{\alpha_{T+1}}.\label{online_op:dyz}
\end{align}

Summing (\ref{online_op:theoremregequ1}) over $t\in[T]$, using \eqref{online_op:dyz}, choosing $\mu_i=\bm{0}_{m_i}$, and setting $y_{T+1}=y_{T}$ gives
\begin{align}\label{online_op:theoremregequ2}
&\NetReg(\{x_{i,t}\},y_{[T]})+\frac{1}{n}\sum_{t=1}^T\sum_{i=1}^n
\Delta_{i,t+1}(\bm{0}_{m_i})\nonumber\\
&\le \varepsilon_1\sum_{t=1}^T\gamma_{t+1}
+\frac{1}{n}\sum_{t=1}^{T}\sum_{i=1}^n\tilde{\Delta}_{i,t+1}(\bm{0}_{m_i})
+\frac{2R(\mathbb{X})^2}{\alpha_{T+1}}
+\frac{2R(\mathbb{X})}{\alpha_{T}}P_T.
\end{align}

To get \eqref{online_op:theoremregequ}, we then establish a lower bound for $\sum_{t=1}^T\sum_{i=1}^n
\Delta_{i,t+1}(\bm{0}_{m_i})$ and an upper bound for $\sum_{t=1}^{T}\sum_{i=1}^n\tilde{\Delta}_{i,t+1}(\bm{0}_{m_i})$.

\noindent {\bf (i-1)} Establish a lower bound for $\sum_{t=1}^T\sum_{i=1}^n
\Delta_{i,t+1}(\bm{0}_{m_i})$.

For any $T\in\mathbb{N}_+$, we have
\begin{align}\label{online_op:delta}
\sum_{t=1}^T\Delta_{i,t+1}(\mu_i)
&=\frac{1}{2}\sum_{t=1}^T\Big(\frac{\|q_{i,t+1}-\mu_i\|^2}{\gamma_{t+1}}
-\frac{\|q_{i,t}-\mu_i\|^2}{\gamma_{t}}
+\Big(\frac{1}{\gamma_{t}}-\frac{1}{\gamma_{t+1}}+\beta_{t+1}\Big)
\|q_{i,t}-\mu_i\|^2\Big)\nonumber\\
&=\frac{\|q_{i,T+1}-\mu_i\|^2}{2\gamma_{T+1}}-\frac{\|\mu_i\|^2}{2\gamma_1}
+\frac{1}{2}\sum_{t=1}^T\Big(\frac{1}{\gamma_{t}}-\frac{1}{\gamma_{t+1}}+\beta_{t+1}\Big)
\|q_{i,t}-\mu_i\|^2.
\end{align}

Substituting $\mu_i=\bm{0}_{m_i}$ into \eqref{online_op:delta} yields
\begin{align}\label{online_op:delta_zero}
\sum_{t=1}^T\Delta_{i,t+1}(\bm{0}_{m_i})
\ge\frac{1}{2}\sum_{t=1}^T\Big(\frac{1}{\gamma_{t}}-\frac{1}{\gamma_{t+1}}+\beta_{t+1}\Big)
\|q_{i,t}\|^2.
\end{align}

\noindent {\bf (i-2)} Establish an upper bound for $\sum_{t=1}^{T}\sum_{i=1}^n\tilde{\Delta}_{i,t+1}(\bm{0}_{m_i})$.

We have
\begin{align}
\sum_{t=1}^{T}\sum_{s=1}^{t-2}\lambda^{t-s-2}\sum_{j=1}^n\|\epsilon^x_{j,s}\|
&=\sum_{t=1}^{T-2}\sum_{j=1}^n\|\epsilon^x_{j,t}\|\sum_{s=0}^{T-t-2}\lambda^s
\le\frac{1}{1-\lambda}\sum_{t=1}^{T-2}\sum_{j=1}^n\|\epsilon^x_{j,t}\|.
\label{online_op:xitxbarT1}
\end{align}

From \eqref{online_op:xitxbar} and \eqref{online_op:xitxbarT1}, for any $\mu_i\in\mathbb{R}^{m_i}$, and $a>0$, we have
\begin{align}
\sum_{t=1}^{T}\sum_{i=1}^n\|\mu_i\|\|x_{i,t}-\bar{x}_{t}\|
&\le \varepsilon_2\sum_{i=1}^n\|\mu_i\|
+\frac{1}{n}\sum_{t=2}^{T}\sum_{i=1}^n\sum_{j=1}^n\|\epsilon^x_{j,t-1}\|\|\mu_i\|
\nonumber\\
&\quad+\sum_{t=2}^{T}\sum_{i=1}^n\|\epsilon^x_{i,t-1}\|\|\mu_i\|
+\frac{\tau}{1-\lambda}\sum_{t=1}^{T-2}\sum_{i=1}^n
\sum_{j=1}^n\|\epsilon^x_{j,t}\|\|\mu_i\|\nonumber\\
&\le \varepsilon_2\sum_{i=1}^n\|\mu_i\|
+\frac{1}{n}\sum_{t=2}^{T}\sum_{i=1}^n\sum_{j=1}^n\|\epsilon^x_{i,t-1}\|\|\mu_j\|\nonumber\\
&\quad+\sum_{t=2}^{T}\sum_{i=1}^n\|\epsilon^x_{i,t-1}\|\|\mu_i\|
+\frac{\tau}{1-\lambda}\sum_{t=2}^{T}\sum_{i=1}^n
\sum_{j=1}^n\|\epsilon^x_{i,t-1}\|\|\mu_j\|\nonumber\\
&\le \varepsilon_2\sum_{i=1}^n\|\mu_i\|
+\frac{1}{n}\sum_{t=2}^{T}\sum_{i=1}^n\sum_{j=1}^n\Big(\frac{1}{4aF_2\alpha_t}\|\epsilon^x_{i,t-1}\|^2
+aF_2\alpha_t\|\mu_j\|^2\Big)\nonumber\\
&\quad
+\sum_{t=2}^{T}\sum_{i=1}^n\Big(\frac{1}{4aF_2\alpha_t}\|\epsilon^x_{i,t-1}\|^2
+aF_2\alpha_t\|\mu_i\|^2\Big)\nonumber\\
&\quad
+\sum_{t=2}^{T}\sum_{i=1}^n\sum_{j=1}^n\Big(\frac{1}{2anF_2\alpha_t}\|\epsilon^x_{i,t-1}\|^2
+\frac{anF_2\tau^2\alpha_t}{2(1-\lambda)^2}\|\mu_j\|^2\Big)\nonumber\\
&= \varepsilon_2\sum_{i=1}^n\|\mu_i\|\
+\sum_{t=2}^{T}\sum_{i=1}^n\Big(a\varepsilon_3\alpha_t\|\mu_i\|^2
+\frac{1}{aF_2\alpha_t}\|\epsilon^x_{i,t-1}\|^2\Big).
\label{online_op:xitxbarT}
\end{align}

For any $\mu_i\in\mathbb{R}^{m_i}$ and $a>0$, we have
\begin{align}
\|\mu_i\|\|x_{i,t}-x_{i,t+1}\|
&\le\|\mu_i\|\|x_{i,t}-z_{i,t+1}\|+\|\mu_i\|\|z_{i,t+1}-x_{i,t+1}\|\nonumber\\
&\le\|\mu_i\|\|x_{i,t}-z_{i,t+1}\|+\frac{1}{aF_2\alpha_{t+1}}\|\epsilon^x_{i,t}\|^2
+\frac{aF_2\alpha_{t+1}}{4}\|\mu_i\|^2.\label{online_op:fxx}
\end{align}

From (\ref{online_op:al_z}) and $\sum_{i=1}^n[W_t]_{ij}=\sum_{j=1}^n[W_t]_{ij}=1$,  we have
\begin{align}
\sum_{i=1}^n\|x_{i,t}-z_{i,t+1}\|
&\le \sum_{i=1}^n(\|x_{i,t}-\bar{x}_t\|+\|\bar{x}_t-z_{i,t+1}\|)\nonumber\\
&=\sum_{i=1}^n\Big(\|x_{i,t}-\bar{x}_t\|
+\Big\|\bar{x}_t-\sum_{j=1}^n[W_t]_{ij}x_{j,t}\Big\|\Big)\nonumber\\
&\le\sum_{i=1}^n\|x_{i,t}-\bar{x}_t\|
+\sum_{i=1}^n\sum_{j=1}^n[W_t]_{ij}\|\bar{x}_t-x_{j,t}\|\nonumber\\
&=2\sum_{i=1}^n\|x_{i,t}-\bar{x}_t\|.\label{online_op:xz}
\end{align}

From \eqref{online_op:xitxbarT}--\eqref{online_op:xz}, for any $\mu_i\in\mathbb{R}^{m_i}$, and $a>0$, we have
\begin{align}\label{online_op:xitxtildemu}
\sum_{t=1}^{T}\sum_{i=1}^nF_2\|\mu_i\|\|x_{i,t}-x_{i,t+1}\|
&\le 2F_2\varepsilon_2\sum_{i=1}^n\|\mu_i\|
+\sum_{t=2}^{T}\sum_{i=1}^n\Big(2aF_2\varepsilon_3\alpha_t\|\mu_i\|^2
+\frac{2}{a\alpha_t}\|\epsilon^x_{i,t-1}\|^2\Big)\nonumber\\
&\quad+\sum_{t=1}^{T}\sum_{i=1}^n
\Big(\frac{1}{a\alpha_{t+1}}\|\epsilon^x_{i,t}\|^2
+\frac{aF_2^2\alpha_{t+1}}{4}\|\mu_i\|^2\Big)\nonumber\\
&\le 2F_2\varepsilon_2\sum_{i=1}^n\|\mu_i\|
+\sum_{t=1}^{T}\sum_{i=1}^n\Big(\frac{3}{a\alpha_{t+1}}\|\epsilon^x_{i,t}\|^2
+a\varepsilon_4\alpha_{t}\|\mu_i\|^2\Big).
\end{align}

Choosing $\|\mu_i\|=1$ in \eqref{online_op:xitxbarT} yields
\begin{align}
&\sum_{t=1}^{T}\sum_{i=1}^n2F_2\|x_{i,t}-\bar{x}_{t}\|\le 2nF_2\varepsilon_2
+\sum_{t=2}^{T}\sum_{i=1}^n\Big(2aF_2\varepsilon_3\alpha_t
+\frac{2}{a\alpha_t}\|\epsilon^x_{i,t-1}\|^2\Big).
\label{online_op:xitxbarT2}
\end{align}

Choosing $\|\mu_i\|=1$ in \eqref{online_op:xitxtildemu} yields
\begin{align}\label{online_op:xitxtilde}
&\sum_{t=1}^{T}\sum_{i=1}^nF_2\|x_{i,t}-x_{i,t+1}\|\le 2nF_2\varepsilon_2
+\sum_{t=1}^{T}\sum_{i=1}^n\frac{3}{a\alpha_{t+1}}\|\epsilon^x_{i,t}\|^2
+\sum_{t=1}^{T}an\varepsilon_4\alpha_{t}.
\end{align}

Combining \eqref{online_op:xitxbarT2} and \eqref{online_op:xitxtilde}, and choosing $a=10$ yields
\begin{align}\label{online_op:xitxtilde2}
\sum_{t=1}^{T}\sum_{i=1}^n\tilde{\Delta}_{i,t+1}(\bm{0}_{m_i})\le 4nF_2\varepsilon_2
+\sum_{t=1}^{T}10n\varepsilon_5\alpha_{t}.
\end{align}

\noindent {\bf (i-3)} Combining (\ref{online_op:xitxtilde2}), (\ref{online_op:delta_zero}), and \eqref{online_op:theoremregequ2} yields \eqref{online_op:theoremregequ}.

\noindent {\bf (ii)} We first provide a loose bound for network  cumulative constraint violation.

We have
\begin{align}\label{online_op:mug}
\mu_i^\top[g_{i,t}(x_{i,t})]_+
&=\mu_i^\top[g_{i,t}(x_{j,t})]_+
+\mu_i^\top[g_{i,t}(x_{i,t})]_+-\mu_i^\top[g_{i,t}(x_{j,t})]_+\nonumber\\
&\ge\mu_i^\top[g_{i,t}(x_{j,t})]_+
-\|\mu_i\|\|[g_{i,t}(x_{i,t})]_+-[g_{i,t}(x_{j,t})]_+\|\nonumber\\
&\ge\mu_i^\top[g_{i,t}(x_{j,t})]_+
-\|\mu_i\|\|g_{i,t}(x_{i,t})-g_{i,t}(x_{j,t})\|\nonumber\\
&\ge\mu_i^\top[g_{i,t}(x_{j,t})]_+
-F_2\|\mu_i\|\|x_{i,t}-x_{j,t}\|\nonumber\\
&\ge\mu_i^\top[g_{i,t}(x_{j,t})]_+
-F_2\|\mu_i\|(\|x_{i,t}-\bar{x}_{t}\|+\|x_{j,t}-\bar{x}_{t}\|),
\end{align}
where the second inequality holds since that the projection operator is non-expansive and the third inequality holds due to \eqref{online_op:assfunction:functionLipg}.

Combining \eqref{online_op:theoremregequ1} and \eqref{online_op:mug}, setting $y_t=y$, and summing over $j\in[n]$ yields
\begin{align}\label{online_op:theoremregequ1_g}
&\sum_{i=1}^n\Big(\Delta_{i,t+1}(\mu_i)+\frac{1}{ n}\sum_{j=1}^n\mu_i^\top[g_{i,t}(x_{j,t})]_+
-\frac{1}{2}\beta_{t+1}\|\mu_i\|^2\Big)+\sum_{i=1}^nf_{t}(x_{i,t})
-nf_{t}(y)\nonumber\\
&\le n\varepsilon_1\gamma_{t+1}+\sum_{i=1}^n\hat{\Delta}_{i,t+1}(\mu_i)
+\frac{1}{ n}\check{\Delta}_t+\frac{1}{ 2\alpha_{t+1}}\sum_{i=1}^n(\|y-z_{i,t+1}\|^2-\|y-z_{i,t+2}\|^2),
\end{align}
where
\begin{align*}
\hat{\Delta}_{i,t+1}(\mu_i)
&=F_2\|\mu_i\|\|x_{i,t}-\bar{x}_{i,t}\|+\tilde{\Delta}_{i,t+1}(\mu_i),\\
\check{\Delta}_t&=\sum_{j=1}^n\sum_{i=1}^nF_2\|\mu_i\|\|x_{j,t}-\bar{x}_{t}\|.
\end{align*}

To get \eqref{online_op:theoremconsequ}, we then establish upper bounds for $\sum_{t=1}^{T}\sum_{i=1}^n\hat{\Delta}_{i,t+1}(\mu_i)$ and $\sum_{t=1}^{T}\check{\Delta}_t$.

\noindent {\bf (ii-1)} Establish an upper bound for $\sum_{t=1}^{T}\sum_{i=1}^n\hat{\Delta}_{i,t+1}(\mu_i)$.

Combining \eqref{online_op:xitxbarT} and \eqref{online_op:xitxtildemu}--\eqref{online_op:xitxtilde}, and choosing $a=20$ yields
\begin{align}\label{online_op:xitxtilde2_g}
\sum_{t=1}^{T}\sum_{i=1}^n\hat{\Delta}_{i,t+1}(\mu_i)
&\le 4nF_2\varepsilon_2+\sum_{t=1}^{T}20n\varepsilon_5\alpha_{t}+
3F_2\varepsilon_2\sum_{i=1}^n\|\mu_i\|\nonumber\\
&\quad+\sum_{t=1}^{T}\sum_{i=1}^n20(F_2\varepsilon_3+\varepsilon_4)\alpha_{t}\|\mu_i\|^2 -\sum_{t=1}^{T}\sum_{i=1}^n\frac{1}{20\alpha_{t+1}}\|\epsilon^x_{i,t}\|^2.
\end{align}

\noindent {\bf (ii-2)} Establish an upper bound for $\frac{1}{n}\sum_{t=1}^{T}\check{\Delta}_t$.

From \eqref{online_op:xitxbar} and \eqref{online_op:xitxbarT1}, for any $\mu_j\in\mathbb{R}^{m_j}$, and $a>0$, we have
\begin{align}
&\sum_{t=1}^{T}\sum_{i=1}^n\sum_{j=1}^n\|\mu_j\|\|x_{i,t}-\bar{x}_{t}\|\nonumber\\
&\le n\varepsilon_2\sum_{j=1}^n\|\mu_j\|
+2\sum_{t=2}^{T}\sum_{i=1}^n\sum_{j=1}^n\|\epsilon^x_{i,t-1}\|\|\mu_j\|
+\frac{n\tau}{1-\lambda}\sum_{t=1}^{T-2}\sum_{i=1}^n
\sum_{j=1}^n\|\epsilon^x_{i,t}\|\|\mu_j\|\nonumber\\
&\le n\varepsilon_2\sum_{j=1}^n\|\mu_j\|
+2\sum_{t=2}^{T}\sum_{i=1}^n\sum_{j=1}^n\|\epsilon^x_{i,t-1}\|\|\mu_j\|
+\frac{n\tau}{1-\lambda}\sum_{t=2}^{T}\sum_{i=1}^n
\sum_{j=1}^n\|\epsilon^x_{i,t-1}\|\|\mu_j\|\nonumber\\
&\le n\varepsilon_2\sum_{j=1}^n\|\mu_j\|
+\sum_{t=2}^{T}\sum_{i=1}^n\sum_{j=1}^n\Big(\frac{1}{2aF_2\alpha_t}\|\epsilon^x_{i,t-1}\|^2
+2aF_2\alpha_t\|\mu_j\|^2\Big)\nonumber\\
&\quad
+\sum_{t=2}^{T}\sum_{i=1}^n\sum_{j=1}^n\Big(\frac{1}{2aF_2\alpha_t}\|\epsilon^x_{i,t-1}\|^2
+\frac{an^2F_2\tau^2\alpha_t}{2(1-\lambda)^2}\|\mu_j\|^2\Big)\nonumber\\
&= n\varepsilon_2\sum_{i=1}^n\|\mu_i\|
+\sum_{t=2}^{T}\sum_{i=1}^n\Big(na\varepsilon_3\alpha_t\|\mu_i\|^2
+\frac{n}{aF_2\alpha_t}\|\epsilon^x_{i,t-1}\|^2\Big).
\label{online_op:xitxbarTij}
\end{align}

Choosing $a=20$ in \eqref{online_op:xitxbarTij} yields
\begin{align}
\sum_{t=1}^{T}\frac{1}{n}\sum_{t=1}^{T}\check{\Delta}_t
&\le F_2\varepsilon_2\sum_{i=1}^n\|\mu_i\|
+\sum_{t=2}^{T}\sum_{i=1}^n\Big(20F_2\varepsilon_3\alpha_t\|\mu_i\|^2
+\frac{1}{20\alpha_t}\|\epsilon^x_{i,t-1}\|^2\Big).
\label{online_op:xitxbarTij_g}
\end{align}

\noindent {\bf (ii-3)} Prove \eqref{online_op:theoremconsequ}.

Let $h_{ij}:\mathbb{R}^{m_i}_{+}\rightarrow\mathbb{R}$ be a function defined as
\begin{align}\label{online_op:gc}
h_{ij}(\mu_i)&=\mu_i^\top\sum_{t=1}^T[g_{i,t}(x_{j,t})]_+ -\frac{1}{2}\|u_i\|^2\Big(\frac{1}{\gamma_1}
+\sum_{t=1}^T(\beta_{t}+\varepsilon_6\alpha_{t})\Big).
\end{align}

Then, noting \eqref{online_op:dyz}, (\ref{online_op:delta}), (\ref{online_op:xitxtilde2_g}), \eqref{online_op:xitxbarTij_g}, and \eqref{online_op:gc}, and summing (\ref{online_op:theoremregequ1_g}) over $t\in[T]$ gives
\begin{align}\label{online_op:theoremregequ2_g}
&\frac{1}{2}\sum_{t=1}^T\sum_{i=1}^n
\Big(\frac{1}{\gamma_{t}}-\frac{1}{\gamma_{t+1}}+\beta_{t+1}\Big)
\|q_{i,t}-\mu_i\|^2+\frac{1}{n}\sum_{i=1}^n\sum_{j=1}^nh_{ij}(\mu_i)+n\NetReg(\{x_{i,t}\},\{y\})\nonumber\\
&\le 4nF_2\varepsilon_2+n\sum_{t=1}^T(\varepsilon_1\gamma_{t}+20\varepsilon_5\alpha_{t})
+4F_2\varepsilon_2\sum_{i=1}^n\|\mu_i\|+\frac{2nR(\mathbb{X})^2}{\alpha_{T+1}}.
\end{align}

Noting that $\mu_{ij}^0=\frac{\sum_{t=1}^T[g_{i,t}(x_{j,t})]_+}
{\frac{1}{\gamma_1}
	+\sum_{t=1}^T(\beta_{t}+2\varepsilon_6\alpha_{t})}$ and substituting $\mu_i=\mu_{ij}^0\in\mathbb{R}^m_{+}$ into \eqref{online_op:gc} yields
\begin{align}\label{online_op:gcequ}
h_{ij}(\mu_{ij}^0)=\frac{\|\sum_{t=1}^T[g_{i,t}(x_{j,t})]_+\Big\|^2}
{2(\frac{1}{\gamma_1}
+\sum_{t=1}^T(\beta_{t}+\varepsilon_6\alpha_{t}))}.
\end{align}

From $g_t(x)=\col(g_{1,t}(x),\dots,g_{n,t}(x))$, we have
\begin{align}
\sum_{i=1}^n\sum_{j=1}^n\Big\|\sum_{t=1}^T[g_{i,t}(x_{j,t})]_+\Big\|^2
=\sum_{j=1}^n\Big\|\sum_{t=1}^T[g_{t}(x_{j,t})]_+\Big\|^2.
\end{align}

From the definition of $\mu_{ij}^0$ and (\ref{online_op:ftgtupper-b}), we have
\begin{align}\label{online_op:mu0equ}
\|\mu_i^0\|\le\frac{F_1T}{\frac{1}{\gamma_1}
+\sum_{t=1}^T(\beta_{t}+\varepsilon_6\alpha_{t})}.
\end{align}

From (\ref{online_op:ftgtupper-a}), we have
\begin{align}\label{online_op:ff}
-\NetReg(\{x_{i,t}\},\{y\})\le F_1T.
\end{align}

Substituting $\mu_{ij}=\mu_{ij}^0$ into (\ref{online_op:theoremregequ2_g}), using \eqref{online_op:gcequ}--\eqref{online_op:ff},  and rearranging  terms yields \eqref{online_op:theoremconsequ}.
\end{proof}

We are now ready to prove Theorem~\ref{online_op:corollaryreg}. The proof is to substitute the specially designed parameter sequences in \eqref{online_op:stepsize1} into the bounds provided in Lemma~\ref{online_op:theoremreg}.

\noindent {\bf (i)} For any constant $a\in[0,1)$ and $T\in\mathbb{N}_+$, it holds that
\begin{align}\label{online_op:sequenceupp}
\sum_{t=1}^T\frac{1}{t^a}\le\int_1^T\frac{1}{t^a}dt+1=\frac{T^{1-a}-a}{1-a}\le\frac{T^{1-a}}{1-a}.
\end{align}
From \eqref{online_op:stepsize1} and (\ref{online_op:sequenceupp}), we have
\begin{align}
\sum_{t=1}^T(\varepsilon_1\gamma_{t}+10\varepsilon_5\alpha_{t})
\le\frac{\varepsilon_1}{\kappa}T^{\kappa}+\frac{10\varepsilon_5\alpha_0}{1-\kappa}T^{1-\kappa}.
\label{online_op:corollaryregequ11}
\end{align}

 From \eqref{online_op:stepsize1}, we have
\begin{align}\label{online_op:betatgammat}
\frac{1}{\gamma_{t}}-\frac{1}{\gamma_{t+1}}+\beta_{t+1}
&=\frac{t}{t^\kappa}-\frac{t+1}{(t+1)^\kappa}+\frac{1}{t^\kappa}
=\frac{t+1}{t^\kappa}-\frac{t+1}{(t+1)^\kappa}>0.
\end{align}

Combining  (\ref{online_op:theoremregequ}), (\ref{online_op:corollaryregequ11}), and (\ref{online_op:betatgammat}) yields
\begin{align}\label{online_op:corollaryregequ1_proof}
\NetReg(\{x_{i,t}\},y_{[T]})
&\le 4F_2\varepsilon_2
+\frac{\varepsilon_1}{\kappa}T^\kappa+\frac{10\varepsilon_5\alpha_0}{1-\kappa}T^{1-\kappa}
+\frac{4R(\mathbb{X})^2T^\kappa}{\alpha_0}+\frac{2R(\mathbb{X})T^\kappa P_T}{\alpha_0},
\end{align}
which gives \eqref{online_op:corollaryregequ1}.

\noindent {\bf (ii)}  From \eqref{online_op:stepsize1} and (\ref{online_op:sequenceupp}), we have
\begin{align}
\sum_{t=1}^T(\varepsilon_1\gamma_{t}+20\varepsilon_5\alpha_{t})
&\le\frac{\varepsilon_1}{\kappa}T^{\kappa}+\frac{20\varepsilon_5\alpha_0}{(1-\kappa)}T^{1-\kappa},
\label{online_op:corollaryregequ11_g}\\
\sum_{t=1}^T(\beta_{t}+\varepsilon_6\alpha_{t})
&\le\frac{1+\varepsilon_6\alpha_0}{1-\kappa}T^{1-\kappa}.\label{online_op:corollaryregequ12}
\end{align}
Combining (\ref{online_op:theoremconsequ}), (\ref{online_op:betatgammat}), \eqref{online_op:corollaryregequ11_g}, and  (\ref{online_op:corollaryregequ12}) yields
\begin{align}\label{online_op:corollaryconsequ_proof}
&\Big(\frac{1}{n}\sum_{j=1}^n\Big\|\sum_{t=1}^T[g_{t}(x_{j,t})]_+\Big\|\Big)^2
\le\frac{1}{n}\sum_{j=1}^n\Big\|\sum_{t=1}^T[g_{t}(x_{j,t})]_+\Big\|^2\nonumber\\
&\le 4n\varepsilon_2F_1F_2T
+2n\Big(1+\frac{1+\varepsilon_6\alpha_0}{1-\kappa}T^{1-\kappa}\Big)
\Big(F_1T +\frac{\varepsilon_1}{\kappa}T^\kappa+\frac{20\varepsilon_5\alpha_0}{1-\kappa}T^{1-\kappa}
+\frac{4R(\mathbb{X})^2T^\kappa}{\alpha_0}\Big).
\end{align}
Combining \eqref{online_op:corollaryconsequ_proof} and
\begin{align}\label{online_op:corollaryconsequ_proof2}
&\sum_{t=1}^T\|[g_{t}(x_{j,t})]_+\|\le\sum_{t=1}^T\|[g_{t}(x_{j,t})]_+\|_1
=\Big\|\sum_{t=1}^T[g_{t}(x_{j,t})]_+\Big\|_1\le\sqrt{m}\Big\|\sum_{t=1}^T[g_{t}(x_{j,t})]_+\Big\|.
\end{align}
yields \eqref{online_op:corollaryconsequ}.

\subsection{Proof of Theorem~\ref{online_op:corollaryreg_sc}}\label{online_op:corollaryregproof_sc}

In addition to the notations defined in the proof of Theorem~\ref{online_op:corollaryreg}, we also denote $\varepsilon_7=\lceil(\frac{1}{\mu})^{\frac{1}{1-c}}\rceil$, where $\lceil \cdot\rceil$ is the ceiling function.

\noindent {\bf (i)} Under Assumption~\ref{online_op:assstrongconvex}, \eqref{online_op:fxy} can be replaced by
\begin{align}\label{online_op:fxy_sc}
f_{i,t}(x_{i,t})-f_{i,t}(y_t)
&\le F_2\|x_{i,t}-x_{i,t+1}\|-\frac{\mu}{2}\|y_t-x_{i,t}\|^2
+\langle\partial f_{i,t}(x_{i,t}),x_{i,t+1}-y_t\rangle.
\end{align}
Note that compared with \eqref{online_op:fxy}, \eqref{online_op:fxy_sc} has an extra term $-\frac{\mu}{2}\|y_t-x_{i,t}\|^2$.
Then, \eqref{online_op:dyz} can be replaced by
\begin{align}
&\frac{1}{n}\sum_{i=1}^{n}\Big(\frac{1}{\alpha_{t+1}}(\|y_t-z_{i,t+1}\|^2
-\|y_{t+1}-z_{i,t+2}\|^2)
-\mu\|y_t-x_{i,t}\|^2\Big)\nonumber\\
&=\frac{1}{n}\sum_{i=1}^{n}\Big(\frac{1}{\alpha_{t}}\|y_t-z_{i,t+1}\|^2
-\frac{1}{\alpha_{t+1}}\|y_{t+1}-z_{i,t+2}\|^2\nonumber\\
&\quad+\Big(\frac{1}{\alpha_{t+1}}-\frac{1}{\alpha_{t}}\Big)\|y_t-z_{i,t+1}\|^2
-\mu\|y_t-x_{i,t}\|^2\Big)\nonumber\\
&=\frac{1}{n}\sum_{i=1}^{n}\Big(\frac{1}{\alpha_{t}}\|y_t-z_{i,t+1}\|^2
-\frac{1}{\alpha_{t+1}}\|y_{t+1}-z_{i,t+2}\|^2\nonumber\\
&\quad+\Big(\frac{1}{\alpha_{t+1}}-\frac{1}{\alpha_{t}}\Big)
\big\|y_{t}-\sum_{j=1}^{n}[W_t]_{ij}x_{j,t}\|^2
-\mu\|y_t-x_{i,t}\|^2\Big)\nonumber\\
&\le\frac{1}{n}\sum_{i=1}^{n}\Big(\frac{1}{\alpha_{t}}\|y_t-z_{i,t+1}\|^2
-\frac{1}{\alpha_{t+1}}\|y_{t+1}-z_{i,t+2}\|^2\nonumber\\
&\quad+\Big(\frac{1}{\alpha_{t+1}}-\frac{1}{\alpha_{t}}\Big)
\sum_{j=1}^{n}[W_t]_{ij}\|y_t-x_{i,t}\|^2
-\mu\|y_t-x_{i,t}\|^2\Big)\nonumber\\
&=\frac{1}{n}\sum_{i=1}^{n}\Big(\frac{1}{\alpha_{t}}\|y_t-z_{i,t+1}\|^2
-\frac{1}{\alpha_{t+1}}\|y_{t+1}-z_{i,t+2}\|^2
+\Big(\frac{1}{\alpha_{t+1}}-\frac{1}{\alpha_{t}}-\mu\Big)
\|y_t-x_{i,t}\|^2\Big),\label{online_op:dyz_sc}
\end{align}
where the inequality holds due to $\sum_{j=1}^{n}[W_t]_{ij}=1$; and the last equality holds due to $\sum_{i=1}^{n}[W_t]_{ij}=1$.

When $t\ge\varepsilon_7$,  we have
\begin{align}\label{online_op:mu_sc}
\frac{1}{\alpha_{t+1}}
-\frac{1}{\alpha_{t}}-\mu
&=\frac{t+1}{(t+1)^{1-c}}-\frac{t}{t^{1-c}}-\mu
<\frac{1}{t^{1-c}}-\mu\le0.
\end{align}

Similar to the way to get \eqref{online_op:corollaryregequ1_proof}, from \eqref{online_op:stepsize1_sc},  \eqref{online_op:dyz_sc}, and \eqref{online_op:mu_sc},  we have
\begin{align}\label{online_op:corollaryregequ1_sc_proof}
\NetReg(\{x_{i,t}\},\check{x}^*_{T})
&\le 4F_2\varepsilon_2
+\frac{\varepsilon_1}{\kappa}T^\kappa+\frac{10\varepsilon_5}{(1-c)}T^{1-c}
+\frac{1}{n}\sum_{i=1}^{n}\frac{1}{\alpha_{1}}\|y_1-z_{i,2}\|^2\nonumber\\
&\quad
+\frac{1}{n}\sum_{i=1}^{n}\sum_{t=1}^{\varepsilon_7-1}\Big(\frac{1}{\alpha_{t+1}}-\frac{1}{\alpha_{t}}-\mu\Big)
\|y_t-x_{i,t}\|^2\nonumber\\
&\le 4F_2\varepsilon_2
+\frac{\varepsilon_1}{\kappa}T^\kappa+\frac{10\varepsilon_5}{1-c}T^{1-c}
+4(1+(\varepsilon_7-1)[1-\mu]_+)R(\mathbb{X})^2,
\end{align}

Noting that $\kappa\ge1-c$ due to $c\ge1-\kappa$, from \eqref{online_op:corollaryregequ1_sc_proof}, we have \eqref{online_op:corollaryregequ1_sc}.

\noindent {\bf (ii)} Similar to the way to get \eqref{online_op:corollaryconsequ_proof}, from \eqref{online_op:stepsize1_sc}, \eqref{online_op:dyz_sc}, and \eqref{online_op:mu_sc},  we have
\begin{align}\label{online_op:corollaryconsequ_proof_sc}
\Big(\frac{1}{n}\sum_{j=1}^n\Big\|\sum_{t=1}^T[g_{t}(x_{j,t})]_+\Big\|\Big)^2
&\le 4n\varepsilon_2F_1F_2T
+2n\Big(1+\frac{1}{1-\kappa}T^{1-\kappa}+\frac{\varepsilon_6}{1-c}T^{1-c}\Big)
\Big(F_1T\nonumber\\
&\quad+\frac{\varepsilon_1}{\kappa}T^\kappa +\frac{20\varepsilon_5}{1-c}T^{1-c}
+4(1+(\varepsilon_7-1)[1-\mu]_+)R(\mathbb{X})^2\Big).
\end{align}

Noting that $1-\kappa\ge1-c$ due to $c\ge\kappa$, from  \eqref{online_op:corollaryconsequ_proof2} and \eqref{online_op:corollaryconsequ_proof_sc}, we have \eqref{online_op:corollaryconsequ_sc}.

\subsection{Proof of Theorem~\ref{online_op:corollaryreg_bandit}}\label{online_op:corollaryregproof_bandit}
In addition to the notations defined in Appendix~\ref{online_op:corollaryregproof},
we also denote $\hat{f}_{i,t}(x)=\mathbf{E}_{v\in\mathbb{B}^p}[f_{i,t}(x+\delta_{t} v)]$, $[\hat{g}_{i,t}(x)]_+=\mathbf{E}_{v\in\mathbb{B}^p}[[g_{i,t}(x+\delta_{t} v)]_+]$, $\hat{b}_{i,t}=[g_{i,t-1}(x_{i,t-1})]_+ +(\hat{\partial} [g_{i,t-1}(x_{i,t-1})]_+)^\top(x_{i,t}-x_{i,t-1})$,  $\hat{y}_{t}=(1-\xi_t)y_{t}$, $\hat{\varepsilon}_{1}=F_1+2pF_2R(\mathbb{X})$,
$\hat{\varepsilon}_5=2F_2\varepsilon_3+p\varepsilon_4$,
$\hat{\varepsilon}_6=40\hat{\varepsilon}_5$, and $\hat{\mu}_{ij}^0=\frac{\sum_{t=1}^T[g_{i,t}(x_{j,t})]_+}
{\frac{1}{\gamma_1}+\sum_{t=1}^T(\beta_{t}+2\hat{\varepsilon}_6\alpha_{t})}$.

To prove Theorem~\ref{online_op:corollaryreg_bandit}, we need some preliminary results.

We first present some useful properties of the gradient estimators $\hat{\partial}f_{i,t}(\cdot)$ and $\hat{\partial}[g_{i,t}(\cdot)]_+$, which are direct extensions of Lemmas~\ref{online_op:lemma_projection} and \ref{dbco:lemma:uniformsmoothing}.
\begin{lemma}\label{online_op:lemma_gradient}
Assume Assumption~\ref{online_op:assfunction} holds.
Then, $\hat{f}_{i,t}(x)$ and $[\hat{g}_{i,t}(x)]_+$ are convex on $(1-\xi_{t})\mathbb{X}$, and for any $i\in[n]$, $t\in\mathbb{N}_+$ $x\in(1-\xi_{t})\mathbb{X}$, $q\in\mathbb{R}^{m_i}_+$,
\begin{subequations}
\begin{align}
&\partial\hat{f}_{i,t}(x)=\mathbf{E}_{\mathfrak{U}_{t}}[\hat{\partial}f_{i,t}(x)],
\label{dbco:lemma_regretdeltaequ:fsmooth1}\\
&f_{i,t}(x)\le \hat{f}_{i,t}(x)\le f_{i,t}(x)+F_2\delta_{t} ,\label{dbco:lemma_regretdeltaequ:fsmooth2}\\
%&\|\nabla\hat{f}_{i,t}(x)\|\le G,\label{dbco:lemma_regretdeltaequ:fsmooth4}\\
&\|\hat{\partial}f_{i,t}(x)\|\le pF_2,
\label{dbco:lemma_regretdeltaequ:fsmooth3}\\
&\partial[\hat{g}_{i,t}(x)]_+=\mathbf{E}_{\mathfrak{U}_{t}}[\hat{\partial}[g_{i,t}(x)]_+],
\label{dbco:lemma_regretdeltaequ:gsmooth1}\\
&q^\top[g_{i,t}(x)]_+\le q^\top[\hat{g}_{i,t}(x)]_+\le q^\top[g_{i,t}(x)]_++ F_2\delta_{t}\|q\|,\label{dbco:lemma_regretdeltaequ:gsmooth2}\\
&\|\hat{\partial}[g_{i,t}(x)]_+\|\le pF_2,
\label{dbco:lemma_regretdeltaequ:gsmooth4}\\
&\|[\hat{g}_{i,t}(x)]_+\|\le F_1.\label{dbco:lemma_regretdeltaequ:gsmooth3}
\end{align}
\end{subequations}
\end{lemma}

The rest of the proof has similar structure as the proof of Theorem~\ref{online_op:corollaryreg}, but the details are different. Note that Lemma~\ref{online_op:lemma_neterror} still holds since \eqref{online_op:al_z} and \eqref{online_op:al_z_bandit} are the same. Moreover, Lemmas~\ref{online_op:lemma_virtualbound}--\ref{online_op:theoremreg} are respectively replaced by Lemmas~\ref{online_op:lemma_virtualbound_bandit}--\ref{online_op:theoremreg_bandit} in the following.

\begin{lemma}\label{online_op:lemma_virtualbound_bandit}
Suppose Assumptions~\ref{online_op:assfunction}--\ref{online_op:assgraph} hold and $\gamma_{t}\beta_{t}\le1,~t\in\mathbb{N}_+$. For all $i\in[n]$ and $t\in\mathbb{N}_+$, $q_{i,t}$ generated by Algorithm~\ref{online_op:algorithm_bandit} satisfy
\begin{align}
\Delta_{i,t}(\mu_i)&\le 2\hat{\varepsilon}_{1}^2\gamma_t
+q_{i,t-1}^\top \hat{b}_{i,t}-\mu_i^\top[g_{i,t-1}(x_{i,t-1})]_+ 
+\frac{1}{2}\beta_t\|\mu_i\|^2+pF_2\|\mu_i\|\|x_{i,t}-x_{i,t-1}\|.
\label{online_op:gvirtualnorm_bandit}
\end{align}
\end{lemma}
\begin{proof}
From \eqref{online_op:domainupper}, \eqref{online_op:ftgtupper-b}, and \eqref{dbco:lemma_regretdeltaequ:gsmooth4}, we have
\begin{align}
\|\hat{b}_{i,t+1}\|\le \hat{\varepsilon}_{1}.\label{online_op:qg_bandit}
\end{align}

Replacing \eqref{online_op:qg} by \eqref{online_op:qg_bandit} and  following steps similar to those used to prove \eqref{online_op:lemma_virtualboundeqy} yields
\begin{align}
\|\beta_tq_{i,t}\|&\le \hat{\varepsilon}_{1}.\label{online_op:lemma_virtualboundeqy_bandit}
\end{align}

From \eqref{online_op:qg_bandit} and \eqref{online_op:lemma_virtualboundeqy_bandit}, we have
\begin{align}\label{online_op:gvirtualnorm_pf2_bandit}
\|\hat{b}_{i,t}-\beta_tq_{i,t-1}\|\le 2\hat{\varepsilon}_{1}.
\end{align}

From \eqref{dbco:lemma_regretdeltaequ:gsmooth4}, we have
\begin{align}
&-2\gamma_t\mu_i^\top(\hat{\partial} [g_{i,t-1}(x_{i,t-1})]_+)^\top(x_{i,t}-x_{i,t-1})
\le2\gamma_tpF_2\|\mu_i\|\|x_{i,t}-x_{i,t-1}\|.\label{online_op:qmu2_bandit}
\end{align}

Replacing \eqref{online_op:gvirtualnorm_pf2} and \eqref{online_op:qmu2} by \eqref{online_op:gvirtualnorm_pf2_bandit} and \eqref{online_op:qmu2_bandit}, respectively, and following steps similar to those used to prove \eqref{online_op:gvirtualnorm} yields \eqref{online_op:gvirtualnorm_bandit}.
\end{proof}

\begin{lemma}\label{online_op:lemma_regretdelta_bandit}
Suppose Assumptions~\ref{online_op:assfunction}--\ref{online_op:assgraph} hold. For all $i\in[n]$, let $\{x_{i,t}\}$ be the sequences generated by Algorithm \ref{online_op:algorithm_bandit} and $\{y_{t}\}$ be an arbitrary sequence in $\mathbb{X}$, then
\begin{align}\label{online_op:lemma_regretdeltaequ_bandit}
\frac{1}{n}\sum_{i=1}^nf_{t}(x_{i,t})
-f_{t}(y_t)
&\le \frac{1}{n}\sum_{i=1}^nq_{i,t}^\top ([g_{i,t}(y_{t})]_+-\mathbf{E}_{\mathfrak{U}_{t}}[\hat{b}_{i,t+1}])
-\frac{1}{n}\sum_{i=1}^n\frac{\mathbf{E}_{\mathfrak{U}_{t}}[\|\epsilon^x_{i,t}\|^2]}{2\alpha_{t+1}}
\nonumber\\
&\quad+\frac{1}{n}\sum_{i=1}^nF_2(2\|x_{i,t}-\bar{x}_{i,t}\|
+p\mathbf{E}_{\mathfrak{U}_{t}}[\|x_{i,t}-x_{i,t+1}\|])\nonumber\\
&\quad+\frac{1}{ 2n\alpha_{t+1}}\sum_{i=1}^n\mathbf{E}_{\mathfrak{U}_{t}}[\|\hat{y}_t-z_{i,t+1}\|^2
-\|\hat{y}_{t+1}-z_{i,t+2}\|^2\nonumber\\
&\quad+\|\hat{y}_{t+1}-x_{i,t+1}\|^2-\|\hat{y}_t-x_{i,t+1}\|^2]\nonumber\\
&\quad+\frac{1}{n}\sum_{i=1}^nF_2(R(\mathbb{X})\xi_{t}+ \delta_{t})(\|q_{i,t}\|+1).
\end{align}
\end{lemma}
\begin{proof}
From $q_{i,t}\ge{\bf 0}_{m_i}$, $x_{i,t},\hat{y}_{t}\in(1-\xi_t)\mathbb{X}$, \eqref{online_op:domainupper}, \eqref{online_op:assfunction:functionLipg}, and \eqref{dbco:lemma_regretdeltaequ:gsmooth2}, we have
\begin{align}
q_{i,t}^\top [\hat{g}_{i,t}(x_{i,t})]_+&\ge q_{i,t}^\top [g_{i,t}(x_{i,t})]_+,\label{dbco:lemma_regretdelta:equ1_g2}\\
q_{i,t}^\top[\hat{g}_{i,t}(\hat{y}_{t})]_+&\le q_{i,t}^\top[g_{i,t}(\hat{y}_{t})]_++ F_2\delta_{t}\|q_{i,t}\|\nonumber\\
&=q_{i,t}^\top([g_{i,t}(y_{t})]_++[g_{i,t}(\hat{y}_{t})]_+-[g_{i,t}(y_{t})]_+)+ F_2\delta_{t}\|q_{i,t}\|\nonumber\\
&\le q_{i,t}^\top[g_{i,t}(y_{t})]_++F_2\|\hat{y}_{t}-y_{t}\|\|q_{i,t}\|+ F_2\delta_{t}\|q_{i,t}\|\nonumber\\
&\le q_{i,t}^\top[g_{i,t}(y_{t})]_++F_2(R(\mathbb{X})\xi_{t}+ \delta_{t})\|q_{i,t}\|.\label{dbco:lemma_regretdelta:equ1_g}
\end{align}

From \eqref{online_op:domainupper}, \eqref{online_op:assfunction:functionLipf}, and \eqref{dbco:lemma_regretdeltaequ:fsmooth2}, we have
\begin{align}
\hat{f}_{i,t}(\hat{y}_{t})-f_{i,t}(y_{t})
&=f_{i,t}(\hat{y}_{t})-f_{i,t}(y_{t})
+\hat{f}_{i,t}(\hat{y}_{t})-f_{i,t}(\hat{y}_{t})\nonumber\\
&\le F_2\|\hat{y}_{t}-y_{t}\|+\hat{f}_{i,t}(\hat{y}_{t})-f_{i,t}(\hat{y}_{t})\nonumber\\
&\le F_2R(\mathbb{X})\xi_{t}+ F_2\delta_{t}.\label{dbco:lemma_regretdelta:equ1}
\end{align}

From $x_{i,t}\in(1-\xi_t)\mathbb{X}$, \eqref{online_op:lxit}, and \eqref{dbco:lemma_regretdeltaequ:fsmooth2}, we have
\begin{align}\label{online_op:lxit_bandit}
\frac{1}{n}\sum_{i=1}^{n}f_t(x_{i,t})
\le\frac{1}{n}\sum_{i=1}^{n}\hat{f}_{i,t}(x_{i,t})+\frac{2F_2}{n}\sum_{i=1}^{n}
\|x_{i,t}-\bar{x}_{t}\|.
\end{align}

We have
\begin{align}\label{online_op:fxy_bandit}
\hat{f}_{i,t}(x_{i,t})-\hat{f}_{i,t}(\hat{y}_t)&\le\langle\partial \hat{f}_{i,t}(x_{i,t}),x_{i,t}-\hat{y}_t\rangle\nonumber\\
&=\langle\mathbf{E}_{\mathfrak{U}_{t}}[\hat{\partial} f_{i,t}(x_{i,t})],x_{i,t}-\hat{y}_t\rangle\nonumber\\
&=\mathbf{E}_{\mathfrak{U}_{t}}[\langle\hat{\partial}f_{i,t}(x_{i,t}),
x_{i,t}-\hat{y}_t\rangle]\nonumber\\
&=\mathbf{E}_{\mathfrak{U}_{t}}[\langle\hat{\partial}f_{i,t}(x_{i,t}),x_{i,t}-x_{i,t+1}\rangle
+\langle\hat{\partial}f_{i,t}(x_{i,t}),x_{i,t+1}-\hat{y}_t\rangle]\nonumber\\
&\le \mathbf{E}_{\mathfrak{U}_{t}}[pF_2\|x_{i,t}-x_{i,t+1}\|
+\langle\hat{\partial}f_{i,t}(x_{i,t}),x_{i,t+1}-\hat{y}_t\rangle],
\end{align}
where the first inequality holds since $x_{i,t},\hat{y}_t\in(1-\xi_t)\mathbb{X}$ and $\hat{f}_{i,t}(x)$ is convex on $(1-\xi_{t})\mathbb{X}$ as shown in Lemma~\ref{online_op:lemma_gradient}; the first equality holds due to \eqref{dbco:lemma_regretdeltaequ:fsmooth1}; the second equality holds since $x_{i,t}$ and $\hat{y}_t$ are independent of $\mathfrak{U}_{t}$; and the last inequality holds due to \eqref{dbco:lemma_regretdeltaequ:fsmooth3}.

Similar to the way to get \eqref{online_op:fxy1}, from \eqref{online_op:al_bigomega_bandit}, we have
\begin{align}
\langle\hat{\partial}f_{i,t}(x_{i,t}),x_{i,t+1}-\hat{y}_t\rangle
&=\langle\hat{\omega}_{i,t+1},x_{i,t+1}-\hat{y}_t\rangle
+\langle\hat{\partial}[g_{i,t}(x_{i,t})]_+ q_{i,t},\hat{y}_t-x_{i,t}\rangle\nonumber\\
&\quad+\langle\hat{\partial}[g_{i,t}(x_{i,t})]_+ q_{i,t},x_{i,t}-x_{i,t+1}\rangle.\label{online_op:fxy1_bandit}
\end{align}

We have
\begin{align}\label{online_op:gyxdelta_bandit}
\mathbf{E}_{\mathfrak{U}_{t}}[\langle\hat{\partial}[g_{i,t}(x_{i,t})]_+ q_{i,t},\hat{y}_t-x_{i,t}\rangle]
&=\langle\mathbf{E}_{\mathfrak{U}_{t}}[\hat{\partial}[g_{i,t}(x_{i,t})]_+] q_{i,t},\hat{y}_t-x_{i,t}\rangle\nonumber\\
&=\langle\partial[\hat{g}_{i,t}(x_{i,t})]_+ q_{i,t},\hat{y}_t-x_{i,t}\rangle\nonumber\\
&\le q_{i,t}^\top [\hat{g}_{i,t}(\hat{y}_{t})]_+ -q_{i,t}^\top [\hat{g}_{i,t}(x_{i,t})]_+\nonumber\\
&\le q_{i,t}^\top[g_{i,t}(y_{t})]_+-q_{i,t}^\top [g_{i,t}(x_{i,t})]_++F_2(R(\mathbb{X})\xi_{t}+ \delta_{t})\|q_{i,t}\|,
\end{align}
where the first equality holds since $x_{i,t}$, $q_{i,t}$, and $\hat{y}_t$ are independent of $\mathfrak{U}_{t}$; the last equality holds due to \eqref{dbco:lemma_regretdeltaequ:gsmooth1}; the first inequality holds since $q_{i,t}\ge{\bf 0}_{m_i}$, $x_{i,t},\hat{y}_t\in(1-\xi_t)\mathbb{X}$ and $[\hat{g}_{i,t}(x)]_+$ is convex on $(1-\xi_{t})\mathbb{X}$ as shown in Lemma~\ref{online_op:lemma_gradient}; and the last inequality holds due to \eqref{dbco:lemma_regretdeltaequ:gsmooth2}, \eqref{dbco:lemma_regretdelta:equ1_g2}, and \eqref{dbco:lemma_regretdelta:equ1_g}.

Same as the way to get \eqref{online_op:omgea2}, from  (\ref{online_op:al_x_bandit}), \eqref{dbco:lemma:projection:xy}, and $\hat{y}_t\in(1-\xi_t)\mathbb{X}\subseteq(1-\xi_{t+1})\mathbb{X}$ due to $\xi_t\ge\xi_{t+1}$, we have
\begin{align}
\langle\hat{\omega}_{i,t+1},x_{i,t+1}-\hat{y}_t\rangle
&\le\frac{1}{2\alpha_{t+1}}\Big(\|\hat{y}_t-z_{i,t+1}\|^2-\|\hat{y}_{t+1}-z_{i,t+2}\|^2
-\|\epsilon^x_{i,t}\|^2\nonumber\\
&\quad+\sum_{j=1}^n[W_{t+1}]_{ij}\|\hat{y}_{t+1}-x_{j,t+1}\|^2-\|\hat{y}_t-x_{i,t+1}\|^2
\Big).\label{online_op:omgea2_bandit}
\end{align}

Combining (\ref{dbco:lemma_regretdelta:equ1})--(\ref{online_op:omgea2_bandit}), taking expectation in $\mathfrak{U}_{t}$, summing over $i\in[n]$, and dividing by $n$, and using $\sum_{i=1}^n[W_t]_{ij}=1,~\forall t\in\mathbb{N}_+$ yields
 (\ref{online_op:lemma_regretdeltaequ_bandit}).
\end{proof}

%Similar to the proof of Lemma~\ref{online_op:theoremreg}, from Lemmas~\ref{online_op:lemma_neterror}, \ref{online_op:lemma_virtualbound_bandit}, and \ref{online_op:lemma_regretdelta_bandit}, we have the following results.
\begin{lemma}\label{online_op:theoremreg_bandit}
Suppose Assumptions~\ref{online_op:assfunction}--\ref{online_op:assgraph} hold and $\gamma_{t}\beta_{t}\le1,~t\in\mathbb{N}_+$. For all $i\in[n]$, let $\{x_{i,t}\}$ be the sequences generated by Algorithm~\ref{online_op:algorithm_bandit}. Then, for any comparator sequence $y_{[T]}\in\calX_{T}$,
\begin{align}
\mathbf{E}[\NetReg(\{x_{i,t}\},y_{[T]})]
&\le 2(p+1)F_2\varepsilon_2
+\sum_{t=1}^T(2\hat{\varepsilon}_1^2\gamma_{t}
+10\hat{\varepsilon}_5\alpha_{t})+\frac{2R(\mathbb{X})^2}{\alpha_{T+1}}
+\frac{2R(\mathbb{X})}{\alpha_{T}}P_T\nonumber\\
&\quad-\frac{1}{2n}\sum_{t=1}^T\sum_{i=1}^n\Big(\frac{1}{\gamma_{t}}
-\frac{1}{\gamma_{t+1}}+\beta_{t+1}\Big)\mathbf{E}[\|q_{i,t}\|^2]\nonumber\\
&\quad
+\sum_{t=1}^TF_2(R(\mathbb{X})\xi_{t}+ \delta_{t})\Big(\frac{\hat{\varepsilon}_{1}}{\beta_t}+1\Big)+\sum_{t=1}^T\frac{2R(\mathbb{X})^2(\xi_t-\xi_{t+1})}{\alpha_{t+1}},
\label{online_op:theoremregequ_bandit}\\
\mathbf{E}\Big[\frac{1}{n}\sum_{i=1}^n\Big\|\sum_{t=1}^T[g_{t}(x_{i,t})]_+\Big\|^2\Big]
&\le 2(p+1)n\varepsilon_2F_1F_2T+2\Big(\frac{1}{\gamma_1}
+\sum_{t=1}^T(\beta_{t}+\hat{\varepsilon}_6\alpha_{t}\Big)\Big(nF_1T\nonumber\\
&\quad+\sum_{t=1}^Tn(2\hat{\varepsilon}_1^2\gamma_{t}
+20\hat{\varepsilon}_5\alpha_{t})+\frac{2nR(\mathbb{X})^2}{\alpha_{T+1}}\nonumber\\
&\quad
-\frac{1}{2}\sum_{t=1}^T\sum_{i=1}^n\Big(\frac{1}{\gamma_{t}}
-\frac{1}{\gamma_{t+1}}+\beta_{t+1}\Big)\mathbf{E}[\|q_{i,t}-\hat{\mu}_{ij}^0\|^2]\nonumber\\
&\quad+\sum_{t=1}^TnF_2(R(\mathbb{X})\xi_{t}+ \delta_{t})\Big(\frac{\hat{\varepsilon}_{1}}{\beta_t}+1\Big)\nonumber\\
&\quad+\sum_{t=1}^T\frac{2R(\mathbb{X})^2(\xi_t-\xi_{t+1})}{\alpha_{t+1}}\Big).
\label{online_op:theoremconsequ_bandit}
\end{align}
\end{lemma}
\begin{proof}
From (\ref{online_op:domainupper}) and $\hat{y}_{t}=(1-\xi_t)y_{t}$, we have
\begin{align}\label{online_op:dxy_bandit}
\|\hat{y}_{t+1}-x_{i,t+1}\|^2-\|\hat{y}_t-x_{i,t+1}\|^2
&\le\|\hat{y}_{t+1}-\hat{y}_t\|\|\hat{y}_{t+1}-x_{i,t+1}+\hat{y}_t-x_{i,t+1}\|\nonumber\\
&\le4R(\mathbb{X})\|\hat{y}_{t+1}-\hat{y}_t\|\nonumber\\
&=4R(\mathbb{X})\|(1-\xi_{t+1})(y_{t+1}-y_t)+(\xi_t-\xi_{t+1})y_t\|\nonumber\\
&=4R(\mathbb{X})\|y_{t+1}-y_t\|+4R(\mathbb{X})^2(\xi_t-\xi_{t+1}).
\end{align}

From \eqref{online_op:lemma_virtualboundeqy_bandit}, \eqref{online_op:dxy_bandit}, and Lemmas~\ref{online_op:lemma_neterror}, \ref{online_op:lemma_virtualbound_bandit}, and \ref{online_op:lemma_regretdelta_bandit}, following steps similar to those used to prove Lemma~\ref{online_op:theoremreg}, we know Lemma~\ref{online_op:theoremreg_bandit} holds.
\end{proof}

We are now ready to prove Theorem~\ref{online_op:corollaryreg_bandit}.

\noindent {\bf (i)} Similar to the way to \eqref{online_op:corollaryregequ1_proof}, from \eqref{online_op:stepsize1_bandit}, (\ref{online_op:sequenceupp}), (\ref{online_op:betatgammat}), and \eqref{online_op:theoremregequ_bandit}, we have
\begin{align*}%\label{online_op:corollaryregequ1_proof_bandit}
\mathbf{E}[\NetReg(\{x_{i,t}\},y_{[T]})]
&\le 2(p+1)F_2\varepsilon_2
+\frac{2\hat{\varepsilon}^2_1}{\kappa}T^\kappa+\frac{10\hat{\varepsilon}_5\alpha_0}{(1-\kappa)}T^{1-\kappa}
+\frac{4R(\mathbb{X})^2T^\kappa}{\alpha_0}\nonumber\\
&\quad+F_2(R(\mathbb{X})+r(\mathbb{X}))\Big(\frac{\hat{\varepsilon}_1}{\kappa}T^\kappa+\log(T)\Big)
\nonumber\\
&\quad+\frac{2R(\mathbb{X})^2(2-\kappa)}{\alpha_0(1-\kappa)}+\frac{2R(\mathbb{X})T^\kappa P_T}{\alpha_0},
\end{align*}
which gives \eqref{online_op:corollaryregequ1_bandit}.

\noindent {\bf (ii)} Similar to the way to \eqref{online_op:corollaryconsequ_proof}, from \eqref{online_op:stepsize1_bandit}, (\ref{online_op:sequenceupp}), (\ref{online_op:betatgammat}), and \eqref{online_op:theoremconsequ_bandit}, we have
\begin{align}\label{online_op:corollaryconsequ_proof_bandit}
\mathbf{E}\Big[\Big(\frac{1}{n}\sum_{j=1}^n\Big\|\sum_{t=1}^T[g_{t}(x_{j,t})]_+\Big\|\Big)^2\Big]
&\le 2(p+1)n\varepsilon_2F_1F_2T
+2n\Big(1+\frac{1+\hat{\varepsilon}_6\alpha_0}{1-\kappa}T^{1-\kappa}\Big)
\Big(F_1T\nonumber\\
&\quad +\frac{2\hat{\varepsilon}^2_1}{\kappa}T^\kappa+\frac{20\hat{\varepsilon}_5\alpha_0}{1-\kappa}T^{1-\kappa}
+\frac{4R(\mathbb{X})^2T^\kappa}{\alpha_0}\nonumber\\
&\quad+F_2(R(\mathbb{X})+r(\mathbb{X}))\Big(\frac{\hat{\varepsilon}_1}{\kappa}T^\kappa+\log(T)\Big)\nonumber\\
&\quad+\frac{2R(\mathbb{X})^2(2-\kappa)}{\alpha_0(1-\kappa)}\Big).
\end{align}

Combining \eqref{online_op:corollaryconsequ_proof2} and \eqref{online_op:corollaryconsequ_proof_bandit} yields \eqref{online_op:corollaryconsequ_bandit}.

\subsection{Proof of Theorem~\ref{online_op:corollaryreg_sc_bandit}}\label{online_op:corollaryregproof_sc_bandit}

\noindent {\bf (i)} From Assumption~\ref{online_op:assstrongconvex} and the last part of Lemma~\ref{dbco:lemma:uniformsmoothing}, we know that $\hat{f}_{i,t}$ is strongly convex with constant $\mu>0$ over $(1-\xi_t)\mathbb{X}$. Thus, \eqref{online_op:fxy_bandit} can be replaced by
\begin{align}\label{online_op:fxy_sc_bandit}
\hat{f}_{i,t}(x_{i,t})-\hat{f}_{i,t}(\hat{y}_t)\le \mathbf{E}_{\mathfrak{U}_{t}}[pF_2\|x_{i,t}-x_{i,t+1}\|
+\langle\hat{\partial}f_{i,t}(x_{i,t}),x_{i,t+1}-\hat{y}_t\rangle]-\frac{\mu}{2}\|\hat{y}_t-x_{i,t}\|^2.
\end{align}

Similar to the way to get \eqref{online_op:corollaryregequ1_sc_proof}, from \eqref{online_op:stepsize1_sc_bandit}  and \eqref{online_op:fxy_sc_bandit}, we have
\begin{align}\label{online_op:corollaryregequ1_sc_proof_bandit}
\mathbf{E}[\NetReg(\{x_{i,t}\},\check{x}^*_{T})]
&\le 2(p+1)F_2\varepsilon_2
+\frac{2\hat{\varepsilon}^2_1}{\kappa}T^\kappa+\frac{10\hat{\varepsilon}_5}{1-c}T^{1-c}\nonumber\\
&\quad+4(1+(\varepsilon_7-1)[1-\mu]_+)R(\mathbb{X})^2\nonumber\\
&\quad+F_2(R(\mathbb{X})+r(\mathbb{X}))\Big(\frac{\hat{\varepsilon}_1}{\kappa}
T^\kappa+\log(T)\Big)+\frac{2R(\mathbb{X})^2(2-\kappa)}{\alpha_0(1-\kappa)},
\end{align}
Noting that $\kappa\ge1-c$ due to $c\ge1-\kappa$, from \eqref{online_op:corollaryregequ1_sc_proof_bandit}, we have \eqref{online_op:corollaryregequ1_sc_bandit}.

\noindent {\bf (ii)} Similar to the way to get \eqref{online_op:corollaryconsequ_proof_sc}, from \eqref{online_op:stepsize1_sc_bandit}  and \eqref{online_op:fxy_sc_bandit}, we have
\begin{align}\label{online_op:corollaryconsequ_proof_sc_bandit}
\mathbf{E}\Big[\Big(\frac{1}{n}\sum_{j=1}^n\Big\|\sum_{t=1}^T[g_{t}(x_{j,t})]_+\Big\|\Big)^2\Big]
&\le 2(p+1)n\varepsilon_2F_1F_2T\nonumber\\
&\quad
+2n\Big(1+\frac{1}{1-\kappa}T^{1-\kappa}+\frac{\hat{\varepsilon}_6}{1-c}T^{1-c}\Big)
\Big(F_1T+\frac{2\hat{\varepsilon}^2_1}{\kappa}T^\kappa \nonumber\\
&\quad +\frac{20\hat{\varepsilon}_5}{1-c}T^{1-c}
+4(1+(\varepsilon_7-1)[1-\mu]_+)R(\mathbb{X})^2\nonumber\\
&\quad+F_2(R(\mathbb{X})+r(\mathbb{X}))\Big(\frac{\hat{\varepsilon}_1}{\kappa}
T^\kappa+\log(T)\Big)+\frac{2R(\mathbb{X})^2(2-\kappa)}{\alpha_0(1-\kappa)}\Big).
\end{align}

Noting that $1-\kappa\ge1-c$ due to $c\ge\kappa$, from  \eqref{online_op:corollaryconsequ_proof2} and \eqref{online_op:corollaryconsequ_proof_sc_bandit}, we have \eqref{online_op:corollaryconsequ_sc_bandit}.
\end{document}